\documentclass[10pt,a4paper]{article}
\usepackage[utf8]{inputenc}
\usepackage{amssymb}
\usepackage{graphicx}
\usepackage{amsmath, amsthm}
\usepackage{amsfonts, dsfont}
\usepackage[comma]{natbib}
\usepackage{color}
\usepackage[dvipsnames]{xcolor}
\usepackage{comment}
\usepackage{mathtools}
\usepackage{subfigure}
\usepackage[ruled,vlined]{algorithm2e}
\mathtoolsset{showonlyrefs}
\usepackage{enumitem}

\newtheorem{corr}{Corollary}
\newtheorem{remark}{Remark}
\newtheorem{lemma}{Lemma}
\newtheorem{prop}{Proposition}
\newtheorem{theorem}{Theorem}

\author{Julien Chevallier\footnote{Universit\'e Grenoble Alpes, LJK UMR-CNRS 5224, 
E-mail: julien.chevallier1@univ-grenoble-alpes.fr}, Anna Melnykova\footnote{Universit\'e de Cergy-Pontoise, AGM UMR-CNRS 8088, 
Universit\'e Grenoble Alpes, LJK UMR-CNRS 5224, 
E-mail: anna.melnykova@u-cergy.fr}, Irene Tubikanec\footnote{Institute for Stochastics, Johannes Kepler University Linz,
E-mail: irene.tubikanec@jku.at}}
\title{Theoretical analysis and simulation methods for Hawkes processes and their diffusion approximation 
}
\date{}

\begin{document}
\maketitle 
\begin{abstract}
    Oscillatory systems of interacting Hawkes processes with Erlang memory kernels were introduced in \cite{Ditlevsen2017eva}. They are piecewise deterministic Markov processes (PDMP) and can be approximated by a stochastic diffusion. First, a strong error bound between the PDMP and the diffusion is proved. Second, moment bounds for the resulting diffusion are derived. Third, approximation schemes for the diffusion, based on the numerical splitting approach, are proposed. These schemes are proved to converge with mean-square order $1$ and to preserve the properties of the diffusion, in particular the hypoellipticity, the ergodicity and the moment bounds. Finally, the PDMP and the diffusion are compared through numerical experiments, where the PDMP is simulated with an adapted thinning procedure.
\end{abstract}

\textbf{Keywords:} Piecewise deterministic Markov processes, Hawkes processes, stochastic differential equations, diffusion processes, neuronal models, numerical splitting schemes

\textbf{Classification:} 
60H35, 65C20, 65C30, 60G55, 60J25

\section*{Introduction} 

Fast and accurate simulation of a biological neuronal network is one of the most extensively studied problems in computational neuroscience. The general goal is to understand how information is processed and transmitted in the brain. 
One of the widely used approaches is to assume that the spike occurrences in a network are described by a point process. Poisson processes, as ``memory less" Markovian processes, can neither take into account a refractory period between two consecutive spikes nor the interaction between neurons, and are thus no proper candidates. Therefore, it is common to model the neuronal activity with Hawkes processes, which are self-exciting point processes with a memory \citep{chevallier2015microscopic,Chornoboy:88, Johnson:96, Pernice:11, Bouret:14}.
The price to pay for using Hawkes processes to model spiking activity is that their investigation is more difficult, since the Markovian theory cannot be directly applied. 

However, for a certain type of memory kernels, so-called Erlang kernels, the dynamics of the point process can be described by a piece-wise deterministic Markov process (PDMP), whose dimension is determined by the ``memory length" of the underlying Hawkes process \citep{Ditlevsen2017eva}. 
This PDMP, also called ``Markovian cascade of successive memory terms" in the literature, is a convenient framework to study the long-time behaviour of the particle system. In particular, it is proved that it is positive Harris recurrent
 and converges to its unique invariant measure exponentially fast in Wasserstein distance \cite[Theorems 1 and 2]{Duarte2019}.

This Markovian cascade and its associated point process can be simulated thanks to the thinning procedure \citep{ogata1981on}, which is a common way to simulate general point processes even without any Markovian assumption. The only requirement in order to apply this method is to provide an upper-bound for the spiking rate of the neurons, which is highly related to the model under consideration \citep{dassios2013exact,Duarte2019}. This procedure yields an exact simulation algorithm but is costly to compute, especially when the number of neurons is large. This results from the fact that the computation time scales linearly with the number of neurons.

In the brain, neurons are clustered in populations with similar behaviours (excitatory, inhibitory, etc). When the network size grows, but the proportion of neurons in each population remains constant, the Markovian cascade can be approximated by a stochastic differential equation (SDE) of the same dimension. In other words, the diffusion approximation theory allows to replace the stochastic term, described by jumps in the PDMP, by a multi-dimensional Brownian motion. 
Passing from a Hawkes process to a diffusion process  substantially simplifies the analysis of the system behaviour.
In particular, the simulation of the diffusion process is much less computationally expensive than that of the Markovian cascade, especially when the number of neurons is large. This results from the fact that the computational time for the SDE does not depend on the number of neurons.
However, the SDE cannot be solved explicitly, and thus the construction of a reliable approximation scheme is required. 

Note that the main difficulty does not lie in the construction of convergent numerical schemes.
For example, standard methods such as the Euler-Maruyama or Milstein schemes converge in the mean-square sense when the time discretization step tends to zero. 
In practice, however, the solution is approximated with a strictly positive time step. As a consequence, even if the discrete solution is known to converge to the continuous process as the time step tends to zero, it does not imply that both processes share the same properties for a fixed discretization step. Thus, the approximation scheme should not be used to study the behaviour of the original model without further analysis of its qualitative properties. Constructing approximation schemes, which are not only convergent, but also preserve the properties of the model, constitutes the main difficulty.

In our case, the first challenge is that the diffusion term of the SDE is highly degenerate and that frequently applied numerical schemes, such as the Euler-Maruyama method, do not preserve the ``propagation of noise property" (formally known as \emph{hypoellipticity}). Second, standard integrators may also fail in preserving second moment properties (see \cite{Mattingly2002}), especially 
when the equation describes oscillatory dynamics, which is the case here. For example, \cite{Higham2004} prove that the Euler-Maruyama method does not preserve the second moment of linear stochastic oscillators. It is expected that this and similar negative results also extend to higher-dimensional and non-linear stochastic oscillators, see, e.g., \cite{Ableidinger2017}. 
Even if higher-order Taylor approximation schemes may solve the problem of degenerate noise structure, they got two major drawbacks. They highly depend on the dimension of the system (which is determined by a parameter in our model) and they commonly fail in preserving ergodic properties.

To overcome these problems, we construct numerical schemes based on the so-called splitting approach. This approach was first developed for ordinary differential equations (ODEs). We refer to \cite{Blanes2009} and \cite{Mclachlan2002} for an exhaustive discussion. 
For an extension to SDEs, see, e.g., \cite{Ableidinger2016,Ableidinger2017,Brehier2018,Leimkuhler2015,Leimkuhler2016,Milstein2004,Misawa2001,Petersen1998,Shardlow2003}.
The main idea of the numerical splitting approach is to decompose the system into explicitly solvable subequations and to find a proper composition of the derived explicit solutions. Such methods usually preserve the properties of the underlying model through the explicitly solved subparts.

The main contributions of this work can be divided into three steps. First, a strong error bound between the Markovian cascade and the stochastic diffusion is proved. This complements the results presented in \cite{Ditlevsen2017eva, Locherbach2019}. Second, moment bounds of order one and two for the stochastic diffusion are derived. Third, simulation algorithms for the diffusion and the PDMP are provided. 
For the diffusion, two splitting schemes, based on the Lie-Trotter and the Strang approach (\cite{Mclachlan2002,Strang1968}), are proposed. 
They are proved to converge with order one in the mean-square sense. Moreover, they are proved to preserve the ergodic property of the continuous process and to accurately reconstruct the moment bounds obtained in the second step.
The simulation method for the PDMP is exact and based on the thinning procedure. In order to apply this method, an explicit upper-bound and a sharper one, involving the numerical computation of polynomial roots, are obtained. Their performances, with respect to the parameters of the model, are discussed. 

This paper is organized as follows. In Section \ref{sec:model}, the finite particle system, the corresponding piece-wise deterministic Markov process and the main notations are introduced. Section \ref{sec:stochdiffusion} is devoted to the stochastic diffusion and to its properties. 
Section \ref{sec:simulation_methods} presents the approximation schemes for the stochastic diffusion. 
Section \ref{section:simulation_PDMP} describes the simulation algorithm for the PDMP. 
Finally, Section \ref{sec:simulation_study} provides a numerical study, illustrating the theoretical results.


\section{Model and notations}\label{sec:model}

The system considered in this paper consists of several populations of neurons, each of them representing a different functional group of neurons (layers in the visual cortex, pools of excitatory and inhibitory neurons in a network, etc.). This system is described by a multivariate counting process, which counts the spike occurrences. In a certain setting, it can be approximated by a stochastic diffusion in the large population limit \citep{Ditlevsen2017eva}. The resulting diffusion is the subject of study in Section \ref{sec:stochdiffusion}.

\subsection{Finite particle system} 

Let us consider a network, consisting of $K$ large populations of neurons, where the number of neurons in the $k$-th population is denoted by $N_k$ and the total number of neurons in the network is $N = N_1 + \dots + N_K$. 
Let $Z^{k, n}_t$ represent the number of spikes of the $n$-th neuron belonging to the $k$-th population during the time interval $[0, t]$.  
The sequence of counting processes $\{(Z^{k, n}_t)_{t\geq 0} ,\, 1\leq k \leq K, \, 1\leq n \leq N_k \}$ is characterized by the intensity processes $(\lambda^{k, n}(t))_{t\geq 0}$, which are formally defined through the relation 
\[
\mathbb{P}(Z^{k, n}_t \text{ has a jump in } (t, t+ dt]|\mathcal{F}_t) = \lambda^{k, n}(t) dt,
\]
where $\mathcal{F}_t$ contains the information about the processes $(Z^{k, n}_t)_{t\geq 0}$ up to time $t$. The mean-field framework considered here corresponds to intensities $\lambda^{k, n}(t)$ given by
\begin{equation}\label{eq:intensity}
\lambda^{k, n}(t) = f_k\left(\sum_{l=1}^K \frac{1}{N_l} \sum_{1\leq m \leq N_l} \int_{(0, t)} h_{kl}(t-s) dZ^{l, m}_s \right),
\end{equation}
where $\{h_{kl}: \mathbb{R}_+ \to \mathbb{R} \}$ is a family of \textit{synaptic weight functions} (also called memory kernels), which model the influence of population $l$ on population $k$. The function $f_k:  \mathbb{R} \to \mathbb{R}_+ $ is the \textit{spiking rate function} of population $k$. 
The expression ``mean-field framework" refers to the fact that the intensity $\lambda^{k, n}(t)$ depends on the whole system only through the ``mean-field" behaviour of each population, namely $\frac{1}{N_l} \sum_{1\leq m \leq N_l} dZ^{l, m}_s$. Furthermore, as $N\to \infty$ we assume that $N_k/N\to p_k>0$ for all $k$. 

Throughout the paper we assume that the functions $f_k$ satisfy the following conditions:
\begin{enumerate}[label=\textbf{(A)}]
    \item  The spiking rate functions $f_k$ are positive, Lipschitz-continuous, non-decreasing and such that $0< f_k \leq f^{\max}_k \text{ for } k = 1,\dots,K$. \label{assumption:a}
\end{enumerate}

In this paper, Erlang-type memory kernels and a cyclic feedback system  of interactions are considered. This means that for each $k$, population $k$ is only influenced by population $k+1$, where we identify $K+1$ with $1$.
In this case, all the memory kernels are null except the ones given by
\begin{equation}\label{eq:memory_kernels}
h_{kk+1}(t) = c_k e^{-\nu_k t}\frac{t^{\eta_k}}{\eta_k!},
\end{equation}
where $c_k = \pm 1$. This constant determines whether the population has an \textit{inhibitory} ($c_k=-1$) or \textit{excitatory} ($c_k=+1$) effect.
The parameter $\eta_k\geq 1$ is an integer number, determining the memory order for the interaction function from population $k+1$ to population $k$. 

The parameters $\eta_k$ and $\nu_k$ determine, intuitively, 
the typical delay of interaction and its time width.
The delay of the influence of the population $k+1$ on population $k$ attains its maximum $\eta_{k+1}/\nu_{k+1}$ units back in time, and its mean is $(\eta_{k+1}+1)/\nu_{k+1}$. 
The larger is this ratio, the more ``old" events are important.
When the ratio is fixed (equal to $\tau$), but both $\eta_k$ and $\nu_k$ tend to infinity, then $h_{kk+1}$ tends to a Dirac mass in $\tau$.
This means that only one specific moment in time is important. The interested reader is referred to \cite{Ditlevsen2017eva} and \cite{Locherbach2019} for more details.

In this paper we are interested in the processes $\{(\bar{X}^{k,1}_t)_{t\geq 0},\, 1\leq k\leq K\}$, which are the arguments of the function $f_k$ in Equation \eqref{eq:intensity} and are defined by
\begin{equation}\label{eq:integrated_intensity}
    \bar{X}^{k,1}_t = \frac{1}{N_{k+1}} \sum_{1\leq m \leq N_{k+1}} \int_{(0, t)} h_{kk+1}(t-s) dZ^{k+1, m}_s.
\end{equation}
When the memory kernels are given in form \eqref{eq:memory_kernels}, the processes defined in \eqref{eq:integrated_intensity} can be obtained as marginals of the process $(\bar{X}_t)_{t\geq 0} = \{ (\bar{X}^{k,j}_t)_{t\geq 0},\, 1\leq k\leq K, 1\leq j\leq \eta_k+1\}$ which solves the following system of dimension $\kappa = \sum_{k=1}^K (\eta_k+1)$:
\begin{equation}\label{eq:PDMP_cascade}
\begin{cases}
d\bar X^{k,j}_t = \left[-\nu_k \bar X^{k,j}_t+\bar X^{k,j+1}_t\right]dt, \: \text{for } j= 1,\dots,\eta_k, \\ 
d\bar X^{k,\eta_k+1}_t = -\nu_k \bar X^{k,\eta_k+1}_tdt+c_k d\bar{Z}^{k+1}_t,\\
\bar{X}_0 = x_0 \in \mathbb{R}^\kappa,
\end{cases}
\end{equation}
where $\bar{Z}^{k+1}_t = \frac{1}{N_{k+1}}\sum_{n=1}^{N_{k+1}} Z^{k+1,n}_t $, each $Z^{k+1,n}_t $ jumping at rate $f(\bar{X}^{k+1,1}_{t-})$, see \cite{Ditlevsen2017eva} for more insight.
This type of equation is called a Markovian cascade in the literature.

The process $(\bar{X}_t)_{t\geq 0}$ summarizes and averages the influence of the past events. This process, along with the firing rate functions $f_k$, determine the dynamics of $(Z^{k,n}_t)_{t\geq 0}$, described by its intensity \eqref{eq:intensity}. 

From a modelling point of view, the process $(\bar{X}^{k,1}_t)_{t\geq 0}$ can be roughly regarded as the voltage membrane potential of any neuron in population $k$. Then, the probability of a neuron to emit a spike is given as a function of its membrane potential.
To summarize, the processes, with coordinates $(k,1)$, defined by \eqref{eq:integrated_intensity}, 
describe the membrane potential in each population, whereas the other coordinates represent higher levels of memory for the process. 

Note that the model presented so far starts with empty memory. The right-hand side of \eqref{eq:intensity} is equal to $f_k(0)$ at time $t=0$ or equivalently $x_0=0$. However, one could easily generalize this to any initial condition $x_0$ in $\mathbb{R}^\kappa$ as it is done in the rest of the paper. Moreover, the interested reader is referred to \cite{Duarte2019}, where a more general model is studied numerically and theoretically for $K=1$ population.

\subsection{Notations}
Now we focus on the case of two interacting populations of neurons ($K = 2$), consisting of $N_1$ and $N_2$ neurons, respectively. Taking $K=2$ allows for an investigation of the interactions between the populations of different sizes while avoiding heavy notations. Throughout the paper the following notation is used: $\mathbb{O}_{n \times m}$ denotes a $n\times m$-dimensional zero matrix and $0_n$ denotes a $n$-dimensional zero vector. 
Then, it is convenient to rewrite system \eqref{eq:PDMP_cascade} in the matrix-vector form
\begin{equation}\label{eq:PDMP_flow}
    d\bar{X}_t = A \bar{X}_t dt + \Gamma \:d\bar{Z}_t,\quad \bar{X}_0 = x_0\in \mathbb{R}^\kappa,
\end{equation}
with
\begin{itemize}
    \item $A \in \mathbb{R}^{\kappa\times\kappa}$ defined as
\begin{equation}\label{eq:A_block_matrix}
A = \left(\begin{matrix}
A_{\nu_1} & \mathbb{O}_{(\eta_1+1) \times (\eta_2+1)} \\ 
\mathbb{O}_{(\eta_2+1) \times (\eta_1+1)} & A_{\nu_2}
\end{matrix}\right),
\end{equation}
where 
 $A_{\nu_k}$ is a $(\eta_k+1)\times (\eta_k+1)$ tri-diagonal matrix with lower-diagonal equal to $0_{\eta_k}$, diagonal equal to $(-\nu_k,\dots, -\nu_k)$ and upper-diagonal equal to $(1,\dots, 1)$,
    \item 
    $\Gamma \in \mathbb{R}^{\kappa\times 2}$ having zero coefficients everywhere, except for $\Gamma_{\eta_1+1,2} = c_1$ and $\Gamma_{\kappa,1} = c_2$, 
    \item and $ \bar{Z}_t = \left( \bar{Z}^1_t,  \bar{Z}^2_t \right)^T$.
\end{itemize}
Throughout the paper the following convention is made. The coordinates of a generic vector $x$ in $\mathbb{R}^\kappa$ are either denoted as $(x_i)_{i=1,\dots,\kappa}$ or $(x^{k,j})_{k=1,2;\, j=1,\dots,\eta_k+1}$ with the relation $i = j$ if $k=1$ and $i = \eta_1+1+j$ if $k=2$. The second notation is usually preferred since each population is easily identified by the index $k$. For some generic function $g:\mathbb{R}^\kappa \to \mathbb{R}^\kappa$, the upper indexes are used as follows: $\left( g(x)\right)^{k,j}$. 
Moreover, it is sometimes more natural to consider some generic $\mathbb{R}^\kappa$-valued process $x_t$ population-wise. Thus, it is split into two components $x^1_t = (x^{1,1}_t,\dots,x^{1,\eta_1+1}_t)\in \mathbb{R}^{\eta_1+1} $ and $x^2_t = (x^{2,1}_t,\dots,x^{2,\eta_2+1}_t)\in \mathbb{R}^{\eta_2+1}$, such that $x_t = (x^1_t, x^2_t)^T \in \mathbb{R}^{\kappa}$.

\section{The limiting stochastic diffusion}\label{sec:stochdiffusion}
In \cite{Ditlevsen2017eva} it is proved that the limit behaviour of \eqref{eq:PDMP_flow} can be approximated by 
the diffusion process $X=(X^1,X^2)^T \in \mathbb{R}^\kappa$,  which is obtained as the the unique strong solution of the SDE
\begin{equation}\label{eq:hawkes_approximation}
dX_t = (AX_t + B(X_t))dt + \frac{1}{\sqrt N}\sigma(X_t)dW_t,  \quad X_0=x_0,
\end{equation}
where $W = (W^{1},W^{2})^T$ is a 2-dimensional Brownian motion, and $x_0\in \mathbb{R}^\kappa$ is a deterministic initial condition. 
The non-linear part of the drift term $B: \mathbb{R}^\kappa \to \mathbb{R}^\kappa$ is given by 
\begin{equation}\label{eq:B_vector}
    B(X)=(B^1(X^2),B^2(X^1))^T,
\end{equation}
where $B^1:\mathbb{R}^{\eta_2+1} \to \mathbb{R}^{\eta_1+1}$ and $B^2:\mathbb{R}^{\eta_1+1} \to \mathbb{R}^{\eta_2+1}$ read as
$B^1(X^2)=(0,\ldots,0,c_1f_2(X^{2,1}))$ and $ B^2(X^1)=(0,\ldots,0,c_2f_1(X^{1,1}))$.
The diffusion component $\sigma: \mathbb{R}^\kappa \to \mathbb{R}^{\kappa\times 2}$ is given by
\begin{equation}\label{eq:diffusion_component}
    \sigma(X)=\begin{pmatrix}
\sigma^1(X^2)  \\ 
\sigma^2(X^1)
\end{pmatrix},
\end{equation}
where
$\sigma^1:\mathbb{R}^{\eta_2+1} \to \mathbb{R}^{(\eta_1+1) \times 2}$ and $\sigma^2:\mathbb{R}^{\eta_1+1} \to \mathbb{R}^{(\eta_2+1) \times 2}$ read as
\begin{equation*}
\sigma^1(X^2)= \begin{pmatrix}
0 & 0 \\
\vdots & \vdots \\
0 & \frac{c_1}{\sqrt{p_2}}\sqrt{f_2(X^{2,1})} 
\end{pmatrix}, \quad 
\sigma^2(X^1)= \begin{pmatrix}
0 & 0 \\
\vdots & \vdots \\
\frac{c_2}{\sqrt{p_1}}\sqrt{f_1(X^{1,1})} & 0
\end{pmatrix}.
\end{equation*}
In other words, the jump term $\Gamma\: d\bar{Z}$, determining the dynamics of the Markovian cascade given in \eqref{eq:PDMP_flow}, is replaced by the sum of a non-linear drift and a diffusion term.

As $N$ goes to infinity, the diffusion term in \eqref{eq:hawkes_approximation} vanishes and the SDE transforms into an ODE of the form 
\begin{equation}\label{eq:ODE}
    dU_t = (AU_t + B(U_t))dt, \quad U_0 = x_0.
\end{equation}
The focus of this paper lies in the theoretical and numerical relations between the PDMP and its stochastic diffusion approximation. Thus, we do not address the properties of ODE \eqref{eq:ODE} in this work and refer to \cite{Ditlevsen2017eva} for related qualitative features and convergence results.

The rest of this section is organized as follows. First, we investigate how accurately the stochastic diffusion approximates the dynamics of the point process, proving a strong error bound between PDMP \eqref{eq:PDMP_flow} and SDE \eqref{eq:hawkes_approximation}. Then, we study the properties of SDE \eqref{eq:hawkes_approximation}, focusing on moment bounds.

\subsection{Strong error bound between the limiting diffusion and the piece-wise deterministic Markov process}
Any error bound of the diffusion approximation is determined by two facts, namely the approximation of a compensated Poisson process by a Brownian motion and the approximation of $N_k$ by  $p_kN$. We get rid of the second approximation by considering SDE \eqref{eq:hawkes_approximation} with $p_k = N_k/N$ and denote the solution of this equation by $Y$. By choosing a different notation we stress the fact that, on the contrary to $X$, it depends on the exact number of  neurons $N_k$ and not on its proportion, obtained in the mean-field limit.
The same convention is used in \cite{Ditlevsen2017eva}, where the following weak error bound is proved. 
\begin{theorem}[\cite{Ditlevsen2017eva}]\label{thm:weak_error_bound}
Grant assumption \ref{assumption:a}  and suppose that all spiking functions $f_k$ belong to the space $C^5_b$ of bounded functions having bounded derivatives up to order 5. Then there exists a constant $C$ depending only on $f_1$, $f_2$ and the bounds on their derivatives such that for all $ \varphi \in C^4_b(\mathbb{R}^\kappa, \mathbb{R})$ and $\forall x_0\in \mathbb{R}^\kappa$,
\begin{equation}\label{eq:result:weak:approx}
\sup_{x\in \mathbb{R}^k} \left| \mathbb{E}\varphi(\bar{X}_t) - \mathbb{E}\varphi(Y_t) \right| \leq Ct \frac{\|\varphi \|_{4,\infty}}{N^2}.
\end{equation}
\end{theorem}
In the following, we strengthen the above result, allowing for a comparison of trajectories of the PDMP and the diffusion. 
\begin{theorem}[Strong error bound]\label{thm:strong:approx:diffusion}
Grant assumption \ref{assumption:a} and let $||\cdot||_{\infty}$ denote the sup norm on $\mathbb{R}^\kappa$. For all $N>0$, a solution $\bar{X}$ of \eqref{eq:PDMP_flow} and a solution $Y$ of \eqref{eq:hawkes_approximation} (with $p_k=N_k/N$) can be constructed on the same probability space such that there exists a constant $C>0$ such that, for all $T>0$, 
\begin{equation}\label{eq:result:strong:approx}
    \sup_{t\leq T} \| \bar{X}_t - Y_t \|_{\infty} \leq \Theta_N e^{CT} \frac{\log(N)}{N}
\end{equation}
almost surely, where $\Theta_N$ is a random variable with exponential moments whose distribution does not depend on $N$. In particular,
\begin{equation}\label{eq:strong:approx:expectation}
\mathbb{E}\left[ \sup_{t\leq T} \| \bar{X}_t - Y_t \|_{\infty} \right] \leq C e^{CT} \frac{\log(N)}{N}.
\end{equation}
\end{theorem} 
 The proof of Theorem \ref{thm:strong:approx:diffusion} is mainly inspired by \cite{Kurtz1978} and relies on two main ingredients, a strong coupling between the standard Poisson process and the Brownian motion and a sharp result on the modulus of continuity for the Brownian motion. All the material is postponed to Appendix \ref{app:strong_bound}.
 
When comparing \eqref{eq:result:weak:approx} and \eqref{eq:strong:approx:expectation}, one notices that there is an exchange between the expectation sign and the absolute value. There are two prices to pay for such an exchange. First, a slower convergence rate with respect to $N$. Second, a faster divergence rate with respect to $t$ (the exponential term is coming from a Grönwall type argument). 
In the following remark we precise the bound on the error which is caused by using directly the parameter $p_k$  instead of $N_k/N$.
\begin{remark}\label{rem:convergence:Nk:pk}
Let $Y$ denote a solution of \eqref{eq:hawkes_approximation} (with parameter $p_k$ equal to $N_k/N$) and $X$ denote a solution of \eqref{eq:hawkes_approximation} with fixed values $p_k$. Following the proof of Theorem \ref{thm:strong:approx:diffusion}, one can show that 
\begin{equation}
    \sup_{t\leq T} \| X_t - Y_t \| \leq \Theta_N e^{CT}\left( \frac{\log(N)}{N} + \max_k\left\{ \frac{1}{\sqrt{p_k N}} \left( 1 - \sqrt{p_k N/N_k}\right) \right\}\right) 
\end{equation}
so that the strong error bound stated in the theorem also holds for the non-modified SDE if $\sqrt{p_k N/N_k} - 1$ is of order $N^{-1/2}$ or of faster order. 

Fortunately, for any fixed $N$, setting $N_1 = \lfloor p_1N \rfloor$ and $N_2 = \lceil p_2 N \rceil$ ensures that $\sqrt{p_k N/N_k} - 1$ is of order $N^{-1} < N^{-1/2}$, which grants that Theorem \ref{thm:strong:approx:diffusion} holds for SDE \eqref{eq:hawkes_approximation}.
\end{remark}

Since SDE \eqref{eq:hawkes_approximation} transforms into ODE \eqref{eq:ODE} as $N$ goes to infinity, the strong error bound can be used to prove the convergence of the PDMP to the solution of the ODE. However, this is beyond the scope of this paper.

\subsection{Properties of the stochastic diffusion} 
The solution process $(X_t)_{t\geq 0}$ of SDE \eqref{eq:hawkes_approximation} is positive Harris recurrent with invariant measure $\pi$ which is of full support (see \cite{Locherbach2019}). It means that the trajectories visit all sets in the support of the invariant measure infinitely often almost surely. More precisely, for any initial condition $x_0$ and measurable set $A$ such that $\pi(A)>0$, $\limsup_{t\to +\infty} \mathbf{1}_{A}(X_t) = 1$ almost surely.
Besides, by following the arguments in \cite{Mattingly2002}, the technical results proven in \cite{Locherbach2019} can be used to prove geometric ergodicity of $(X_t)_{t\geq 0}$ as stated in Proposition \ref{prop:ergodicity_process} below.


In order to state the geometric ergodicity of $(X_t)_{t\geq 0}$, let us first specify the Lyapunov function $G:\mathbb{R}^\kappa\to \mathbb{R}$ introduced in \cite{Ditlevsen2017eva}:
\begin{equation}\label{eq:lyapunov_eva}
G(x) = \sum_{k=1}^K \sum_{j=1}^{\eta_k+1}\frac{j}{\nu^{j-1}_k}J(x^{k,j}),
\end{equation}
where $J$ is some smooth approximation of the absolute value. In particular, $J(x) = |x|$ for all $|x|\geq 1$ and $\max\{ |J^\prime(x)|,|J^{\prime\prime}(x)|\} \leq J_c$ for all $x$, for some finite constant $J_c$.

\begin{prop}[Geometric ergodicity]\label{prop:ergodicity_process}
Grant assumption \ref{assumption:a}. Then the solution of SDE \eqref{eq:hawkes_approximation} has a unique invariant measure $\pi$ on $\mathbb{R}^\kappa$. For all initial conditions $x_0$ and all $m\geq 1$, there exist $C = C(m) >0$ and  $\lambda = \lambda(m)>0$ such that, for all measurable functions $g:\mathbb{R}^\kappa\to \mathbb{R}$ such that $|g|\leq G^m$,
\begin{equation*}
    \forall t \geq 0, \quad \left| \mathbb{E} g(X_t) - \pi(g)\right|\leq C G(x_0)^m e^{-\lambda t}.
\end{equation*}
\end{prop}
\begin{proof}
The proof closely follows that of Theorem 3.2 in \cite{Mattingly2002} and is based on Lyapunov and minorization conditions  (the latter is implied by the existence of a smooth transition density and the irreducibility of the space).  

 (i) First, we use the fact that $G$ is a Lyapunov function for $X$  \cite[Proposition 5]{Ditlevsen2017eva}, i.e., $\exists \alpha, \beta >0,$ s.t. 
 \[
 \mathcal{A}^XG(x) \leq -\alpha G(x) + \beta,
 \]
 where $ \mathcal{A}^XG(x)$ is the infinitesimal generator of \eqref{eq:hawkes_approximation}. 
 
 (ii) Then, we note that, from any initial condition $x_0$, for any time $T>0$ and any open set $O$, the probability that $X_T$ belongs to $O$ is positive.
 It is ensured by the controllability of system \eqref{eq:hawkes_approximation} (see Theorem 4 in \cite{Locherbach2019}).

 (iii) Finally, we note that the process $(X_t)_{t\geq 0}$ possesses a smooth transition density. Its existence is ensured by verifying the H\"ormander condition, which is done in Proposition 7 of \cite{Ditlevsen2017eva}. 

The rest of the proof follows as in the proof of  \cite[Theorem 3.2.]{Mattingly2002}: apply \cite[Theorem 2.5.]{Mattingly2002} to some discrete-time sampling of the process and conclude by interpolation.
\end{proof}
Also note that the rank of the diffusion matrix $\sigma\sigma^T$ is smaller than the dimension of system \eqref{eq:hawkes_approximation}. This means that the system is not elliptic. However, the specific cascade structure of the drift ensures that the noise is propagated through the whole system via the drift term, such that the diffusion is hypoelliptic in the sense of stochastic calculus of variations \citep{Delarue2010,Malliavin2006}.
We also note that SDE \eqref{eq:hawkes_approximation} is semi-linear, with a linear term given by matrix \eqref{eq:A_block_matrix}. Thus, its solution can be written in the form of a convolution equation (see, among others, \citet[Section 3]{Mao2007}).
\begin{prop}\label{prop:convolution}
The unique solution of \eqref{eq:hawkes_approximation} satisfies
\begin{equation}\label{eq:Conv_X}
   X_t = e^{At}x_0 + \int_0^t e^{A(t-s)} B(X_s) ds + \frac{1}{\sqrt{N}} \int_0^t e^{A(t-s)} \sigma(X_s) dW_s.
\end{equation}
\end{prop}
\begin{proof}
Consider the process $Y_t = e^{-At}X_t$. By It\^o's formula one obtains
\begin{eqnarray*}
    d\left(e^{-At}X_t\right) &=& \left( -Ae^{-At}X_t + e^{-At} \left(  AX_t + B(X_t)\right)\right)dt + \frac{e^{-At} }{\sqrt{N}}  \sigma(X_t) dW_t \\ &=&
     e^{-At}  B(X_t)dt + \frac{e^{-At}}{\sqrt{N}}   \sigma(X_t) dW_t.
\end{eqnarray*}
Integrating both parts yields
\[
e^{-At}X_t = x_0 + \int_0^t e^{-As} B(X_s) ds + \frac{1}{\sqrt{N}} \int_0^t e^{-As} \sigma(X_s) dW_s.
\]
Multiplying the expression by $e^{At}$ gives the result. 
\end{proof}
Note that from this form, it is straightforward to see that the diffusion term is of full rank. 
Intuitively, this ensures the hypoellipticity. Further, systems of type \eqref{eq:Conv_X} are called stochastic Volterra equations \citep{Jaber2019}. 

Now we focus on first and second moment bounds. The following results are needed, in particular, to ensure the accuracy of the approximation scheme in Section \ref{sec:simulation_methods}. In the following remark we provide some purely computational results in order to ease the further analysis.
\begin{remark}\label{rem:expA}
Due to the block-structure of the matrix $A$ introduced in \eqref{eq:A_block_matrix}, its matrix exponential $e^{At}$ can be computed as
\begin{equation*}
e^{At} = \left(\begin{matrix}
e^{A_{\nu_1}t} & \mathbb{O}_{(\eta_1+1) \times (\eta_2+1)} \\ 
\mathbb{O}_{(\eta_2+1) \times (\eta_1+1)} & e^{A_{\nu_2}t}
\end{matrix}\right),
\end{equation*}
where $e^{A_{\nu_k}t}$, $k=1,2$, is a $(\eta_k+1)\times (\eta_k+1)$ upper-triangular matrix given by 
\begin{equation}\label{eq:A_exponential}
    e^{A_{\nu_k}t} = e^{-\nu_k t }\left(\begin{matrix}
    1 & t & \frac{t^2}{2} & \dots & \frac{t^{\eta_k}}{\eta_k!} \\ 
    0 & 1 & t &  \dots & \frac{t^{\eta_k-1}}{(\eta_k-1)!} \\ 
     \vdots & \ddots & \ddots & \ddots & \vdots \\ 
     \vdots & \vdots & \ddots & \ddots & \vdots \\ 
     0 & 0 & 0 & \hdots & 1
    \end{matrix}\right).
\end{equation}
In further computations we will often use the vectors $e^{At} X_{s}$. The elements of $e^{At} X_{s}$ are given by the formula
 \begin{equation}\label{eq:jth_element}
     \left(e^{At} X_{s}\right)^{k,j} = e^{-\nu_kt}\sum_{m=j }^{\eta_k+1}\frac{t^{m-j}}{(m-j)!}X^{k,m}_s.
 \end{equation}
\end{remark}
\begin{theorem}[First moment bounds of the diffusion process]\label{thm:moment_bounds_process}
Grant assumption \ref{assumption:a}. The following bounds hold for the components of $\mathbb{E}[X_t]$:
\begin{equation*}
  \mathcal{I}^{k,j}_{\min} \leq  \mathbb{E}[X^{k,j}_t] \leq \mathcal{I}^{k,j}_{\max},
\end{equation*}
where 
\begin{align*}
    \mathcal{I}^{k,j}_{\min} &= \left(e^{At} x_{0}\right)^{k,j} + \left[1-e^{-t\nu_k}\sum_{l=0}^{\eta_k+1-j}\frac{(t\nu_k)^l}{l!}  \right]\min\left\{0, \frac{c_k f_{k+1}^{\max}}{\nu_k^{(\eta_k+2-j)} } \right\}, \\ 
    \mathcal{I}^{k,j}_{\max} &= \left(e^{At} x_{0}\right)^{k,j} + \left[1-e^{-t\nu_k}\sum_{l=0}^{\eta_k+1-j}\frac{(t\nu_k)^l}{l!}  \right]\max\left\{0, \frac{c_k f_{k+1}^{\max}}{\nu_k^{(\eta_k+2-j)} } \right\}. 
\end{align*}
\end{theorem}
\begin{proof}[Proof of Theorem \ref{thm:moment_bounds_process}]
From Proposition \ref{prop:convolution} and Remark \ref{rem:expA}, it follows that the convolution-based representation of the $k$-th population is given by
\begin{align*}
X^{k}_t = (e^{At}x_0)^k+\int_0^t e^{A_{\nu_k}(t-s)} B^k(X^{k+1}_s) ds + \frac{1}{\sqrt{N}} \int_0^t e^{A_{\nu_k}(t-s)} \sigma^k(X^{k+1}_s) dW_s. \label{eq:Conv_X_1} 
\end{align*}
Consequently, the $j$-th components are given by
\begin{eqnarray*}
X^{k,j}_t &=& \underbrace{\left(e^{At} x_{0}\right)^{k,j}}_{:=T_1(t)}+\underbrace{\int_0^t c_k f_{k+1}(X^{{k+1},1}_s)  \frac{e^{-{\nu_k}(t-s)}}{(\eta_k+1-j)!} (t-s)^{\eta_k+1-j} ds}_{:=T_2(t)} \\ &+& \underbrace{\frac{1}{\sqrt{N}} \int_0^t \frac{c_k}{\sqrt{p_{k+1}}}\sqrt{f_{k+1}(X^{{k+1},1}_s)}  \frac{e^{-{\nu_{k}}(t-s)}}{(\eta_{k}+1-j)!} (t-s)^{\eta_{k}+1-j}  dW^{k+1}_s}_{:=T_3(t)}.
\end{eqnarray*}
Note that, $\mathbb{E}[T_1(t)]=T_1(t)$ and $\mathbb{E}[T_3(t)]=0$. It remains to consider $T_2(t)$. The fact that the intensity function is bounded by $0<f_{k+1}\leq f_{k+1}^{\max}$ implies that
\begin{align*}
\min\{0,c_k\}  \frac{f_{k+1}^{\max}}{(\eta_k+1-j)!} I^{k,j}  \leq \mathbb{E}[T^2(t)] \leq \max\{0,c_k\}  \frac{f_{k+1}^{\max}}{(\eta_k+1-j)!} I^{k,j} ,
\end{align*}
where 
\[
I^{k,j} = \int_0^t e^{-{\nu_k}(t-s)}  (t-s)^{\eta_k+1-j} ds. 
\]
Now, let us consider the integral $I^{k,j}$:
\begin{equation*}
     \int_0^t e^{-{\nu_k}(t-s)} (t-s)^{\eta_k+1-j} ds = t^{\eta_k+1 - j } \int_0^t e^{-{\nu_kt}\frac{t-s}{t}} \left(\frac{t-s}{t}\right)^{\eta_k+1-j} ds.
\end{equation*}   
Setting $z = \frac{t-s}{t}$ yields
\begin{equation*} 
t^{\eta_k+2 - j } \int_0^1e^{-{\nu_kt}z} z^{\eta_k+1-j}  dz =   \frac{(\eta_k+1-j)!}{\nu_k^{(\eta_k+2-j)}}\left[1-e^{-\nu_kt}\sum_{l=0}^{\eta_k+1-j}\frac{(\nu_kt)^l}{l!} \right].
\end{equation*}
This gives the result.  
\end{proof}
\begin{remark}\label{rem:bounds_process}
Recalling \eqref{eq:jth_element} and using that $\lim\limits_{t\to\infty} e^{-\nu_kt}\sum_{l=0}^{\eta_k+1-j}\frac{(t\nu_k)^l}{l!}=0$, it follows from Theorem \ref{thm:moment_bounds_process} that
\begin{equation*}
  \min\left\{0, \frac{c_k f_{k+1}^{\max}}{\nu_k^{(\eta_k+2-j)} } \right\} \leq  \lim_{t\to\infty} \mathbb{E}[X^{k,j}_t] \leq \max\left\{0, \frac{c_k f_{k+1}^{\max}}{\nu_k^{(\eta_k+2-j)} } \right\}.
\end{equation*}
\end{remark}
The derived moment bounds give some intuition on how the system behaves in the long run. Remarkably, depending on whether $c_k$ is positive or negative, the trajectories of $(X_t^k)_{t \geq 0}$ are on average bounded by $0$ from below or above, respectively. This is in agreement with the fact that the sign of $c_k$ defines whether the corresponding neural population is excitatory ($c_k = +1$) or inhibitory ($c_k = -1$). 
Moreover, we may immediately see the effect of increasing the memory order $\eta_k$, depending on the constant $\nu_k$. When $\nu_k = 1$, the bounds for all $j$ components are determined entirely by $c_k$ and the bounds of the intensity functions. When $\nu_k < 1$ and $\eta_k \to \infty$, then the first components, presenting the current state of the process, tend to infinity. Similarly, for $\nu_k > 1$, the trajectories are attracted to $0$. Finally, note that the first moment bounds do not depend on the number of neurons in the system. 
\begin{theorem}[Second moment bounds of the diffusion process]\label{thm:second_moment_bounds_process}
Grant assumption \ref{assumption:a}. The following bounds hold for $\mathbb{E}[(X^{k,j}_t)^2]$:
\begin{eqnarray*}
    \mathbb{E}[(X^{k,j}_t)^2] &\leq& \left( \left(e^{At} x_{0}\right)^{k,j} \right)^2 + 2  \left(e^{At} x_{0}\right)^{k,j}\max\left\{0,\frac{c_k f^{\max}_{k+1}}{(\eta_k+1-j)!} I^{k,j}_1(t)\right\} \\ &+& 
    f^{\max}_{k+1}\left(\frac{c_k }{(\eta_k+1-j)!}\right)^2\left(\sqrt{f^{\max}_{k+1}} \: I^{k,j}_1(t) + \sqrt{\frac{I^{k,j}_2(t)}{N\cdot p_{k+1}}} \right)^2,
\end{eqnarray*}
where $I^{k,j}_u(t), \: u = 1,2$, are defined as
\begin{eqnarray*}
    I^{k,j}_u(t) &:=& \int_0^t e^{-u\nu_k(t-s)}(t-s)^{u(\eta_k+1-j)}ds \\ &=& \frac{(u(\eta_k + 1 - j))!}{(u\nu_k)^{u(\eta_k+1-j)+1}} \left[ 1-e^{-ut\nu_k} \sum_{l=0}^{u(\eta_k+1-j)}\frac{(ut\nu_k)^l}{l!} \right].
\end{eqnarray*}
\end{theorem}

 The proof of Theorem \ref{thm:second_moment_bounds_process} is similar to the one of Theorem \ref{thm:moment_bounds_process} and is postponed to Appendix \ref{app:second_moment}. 

\begin{remark}\label{rem:second_moment_asymp_bound_process}
Theorem \ref{thm:second_moment_bounds_process} gives the following asymptotic bounds:
\begin{multline*}
 \lim_{t\to\infty} \mathbb{E}[(X^{k,j}_t)^2] \leq f^{\max}_{k+1}\left(\frac{c_k }{(\eta_k+1-j)!}\right)^2 
 \left(\sqrt{f^{\max}_{k+1}} C^{k,j}_1 + \sqrt{\frac{C^{k,j}_2}{N\cdot p_{k+1}}}\right)^2,
\end{multline*}
where 
\[
C^{k,j}_u := \lim_{t\to\infty}I^{k,j}_u(t)=\frac{(u(\eta_k + 1 - j))!}{(u\nu_k)^{u(\eta_k+1-j)+1}}.
\]
\end{remark}
Note that for $N \to \infty$, the bound obtained in Theorem \ref{thm:second_moment_bounds_process} equals the square of the bound for the first moment, derived in Theorem \ref{thm:moment_bounds_process}. This is in agreement with the fact that the stochastic system \eqref{eq:hawkes_approximation} transforms into an ODE as $N$ increases \citep{Ditlevsen2017eva}. In other words, its diffusion coefficient tends to $0$ as $N$ tends to infinity.
\begin{figure}
\begin{centering}
    \includegraphics[width=1.0\textwidth]{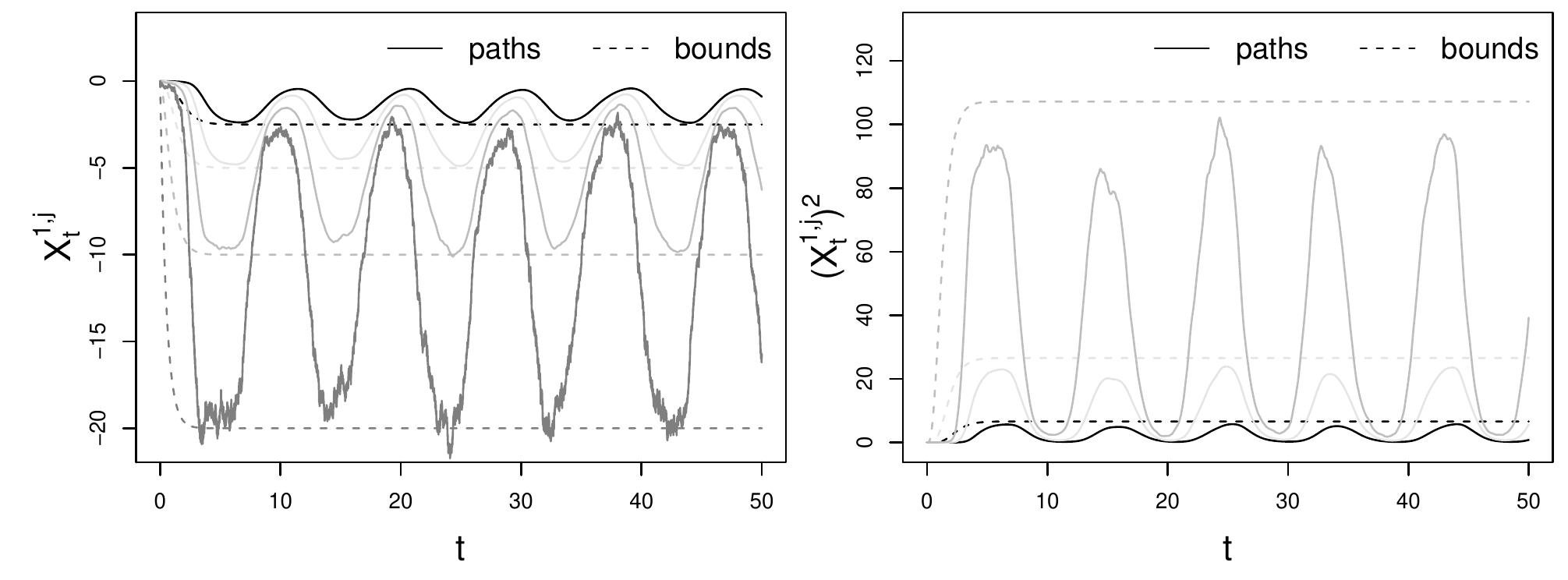}
	\caption{First (left panel) and second (right panel) moment bounds with respective trajectories of the inhibitory population $k=1$. The rate function $f_2$ is given in Section \ref{sec:simulation_study}. The parameters are $\eta_1=3$, $\nu_1=2$, $N=20$ and $p_2=1/2$.}
	\label{fig:moment_bounds}
\end{centering}
\end{figure}

In Figure \ref{fig:moment_bounds}, the first and second moment bounds, derived in Theorem \ref{thm:moment_bounds_process} and Theorem \ref{thm:second_moment_bounds_process}, respectively, are illustrated. In the left panel, we plot 4 sample trajectories (solid lines) of an inhibitory population and their lower first moment bounds (dashed lines). The main variable $X^{k,1}$ and its lower moment bound are depicted in black. The remaining 3 trajectories are auxiliary variables. They (and their corresponding bounds) are depicted in different shades of grey. We see that the trajectories can exceed the theoretical bounds, especially when the effect of noise is large. On average, the trajectories stay within the bounds.
In the right panel, we plot the square of the first 3 components of $X^{k}$ (and their second moment bounds), omitting the 4-th one in order to stay within an easily interpretative scale. We conclude that the bounds are rather precise for the parameter setting under consideration.

\section{Numerical splitting schemes for the stochastic diffusion}\label{sec:simulation_methods}
The solution of system \eqref{eq:hawkes_approximation} cannot be written in an explicit form, and thus a numerical approximation is required. Let $[0,T]$ with $T>0$ be the time interval of interest and consider the discretization  $(t_i)_{i=0,\dots, i_{\max}}$ given by $t_i = i\Delta$, where $\Delta = T/i_{\max}$.
In the following, $\tilde{X}_{t_i}$ denotes a numerical realisation of the diffusion process, evaluated at the discrete time points. 

We derive and investigate two numerical schemes based on the splitting approach. The goal of this method is to divide the equation into explicitly solvable subequations and to compose the obtained explicit solutions in a proper way. Usually, the choice of the subsystems is not unique. Here, because of the specific structure of SDE \eqref{eq:hawkes_approximation}, we split it into the subsystems
\begin{align}
dX^{[1]}_t &= AX^{[1]}_tdt,\label{eq:subsys1} \\ 
dX^{[2]}_t &= B(X^{[2]}_t)dt + \frac{1}{\sqrt N}\sigma(X^{[2]}_t)dW_t. \label{eq:subsys2}
\end{align}
Both subsystems are explicitly solvable. The first one is a linear ODE whose flow is given by $\psi^{[1]}_{t}: x\mapsto e^{A t}x$.
For the second one, recall that $B$ and $\sigma$ are given by \eqref{eq:B_vector} and \eqref{eq:diffusion_component}, respectively. It is easy to see that all components of $X^{[2]}$, except for two ($X^{[2],1,\eta_1+1}$ and $X^{[2],2,\eta_2+1}$) have null derivative. Moreover, the drift and diffusion coefficients of $X^{[2],1,\eta_1+1}$ only depend on $X^{[2],2,1}$ and vice versa.
Hence, the respective explicit (stochastic) flows are given by
\begin{align*}
\psi^{[1]}_{t}(x) &:= e^{At} x,\\ 
\psi^{[2]}_{t}(x) &:= x + t B(x) + \frac{\sqrt{t}}{\sqrt N} \sigma(x)\xi,
\end{align*}
where $\xi=(\xi^1,\xi^2)^T$ is a $2$-dimensional standard normal vector.
Then, the Lie-Trotter and the Strang compositions of flows \citep{Mclachlan2002,Strang1968} are given as follows
\begin{gather}
\label{eq:split_scheme_LT2}
\tilde X^{LT}_{t_{i+1}} = \left(\psi^{[1]}_{\Delta}\circ \psi^{[2]}_{\Delta}\right)\left(\tilde X^{LT}_{t_i}\right) = e^{A\Delta}\left( \tilde X^{LT}_{t_i} + \Delta  B(\tilde X^{LT}_{t_i}) + \frac{\sqrt{\Delta}}{\sqrt N} \sigma(\tilde X^{LT}_{t_i})\xi_i \right),
\end{gather}
\begin{align}
\label{eq:split_scheme}
\tilde X^{ST}_{t_{i+1}} &= \left(\psi^{[1]}_{\frac{\Delta}{2}}\circ \psi^{[2]}_{\Delta}\circ  \psi^{[1]}_{\frac{\Delta}{2}}\right)\left(\tilde X^{ST}_{t_i}\right) \\ &= e^{A\Delta} \tilde X^{ST}_{t_i} + \Delta  e^{A\frac{\Delta}{2}}B(e^{A\frac{\Delta}{2}}\tilde X^{ST}_{t_i}) + \frac{\sqrt{\Delta}}{\sqrt N}e^{A\frac{\Delta}{2}} \sigma(e^{A\frac{\Delta}{2}}\tilde X^{ST}_{t_i})\xi_i,
\end{align}
respectively, with $\tilde{X}^{LT}_{0} = \tilde{X}^{ST}_{0} = x_0$ and $(\xi_i)_{i=1,\dots,i_{\max}}$ i.i.d.
The two splitting schemes \eqref{eq:split_scheme_LT2} and \eqref{eq:split_scheme} define numerical solutions of SDE \eqref{eq:hawkes_approximation}. Note that by setting $\sigma(x)\equiv 0$, both schemes can be used for simulating ODE \eqref{eq:ODE}.

For the sake of simplicity, we focus on the Lie-Trotter splitting \eqref{eq:split_scheme_LT2} in the subsequent analysis, since its representation is more intuitive. Thus, throughout Section \ref{sec:simulation_methods} we set $\tilde{X}\equiv \tilde{X}^{LT}$. However, similar results can be obtained also for the more evolved Strang approach \eqref{eq:split_scheme}.  

\begin{remark}\label{rem:hypoellipticity_scheme}
Note that thanks to the matrix exponential entering the diffusion terms in \eqref{eq:split_scheme_LT2} and \eqref{eq:split_scheme}, the noise propagates through all components of the system at each time step. In other words, the conditional variance matrix $\Sigma$ is of full rank and is given by
\begin{equation*}
    \Sigma\left[\tilde X_{t_{i+1}}|\tilde X_{t_i} \right] := \frac{\Delta}{N} e^{A\Delta}\sigma( \tilde X_{t_i})\sigma^T( \tilde X_{t_i}) \left(e^{A\Delta}\right)^T.
\end{equation*}
This can be regarded as a discrete analogue of the hypoellipticity of the continuous process, a property that the approximation methods based on the It\^o-Taylor expansion of the infinitesimal generator of \eqref{eq:hawkes_approximation} (see \cite{Kloeden2003}) do not preserve. 
\end{remark}
\subsection{Strong convergence in the mean square sense}
Now we focus on the convergence in the mean-square sense and show that the numerical solutions obtained via the splitting approach converge to the process as the time step $\Delta\to 0$ with order $1$. The frequently applied Euler-Maruyama scheme usually converges with mean-square order $1/{2}$ if the noise is multiplicative \citep{Kloeden2003,Milstein2004}, as it is the case for system \eqref{eq:hawkes_approximation}. In the following result, thanks to the specific structure of the noise component, we show that the Euler-Maruyama scheme coincides with the Milstein scheme, which is known to converge with mean-square order 1. This result is then used to establish the convergence order of the splitting scheme.
\begin{theorem}[Mean-square convergence of the splitting scheme]\label{thm:convergence}
Grant assumption \ref{assumption:a}. Let $\tilde{X}_{t_i}$ denote the numerical method defined by \eqref{eq:split_scheme_LT2} at time point $t_i$ and starting from $x_0$. Then $\tilde{X}_{t_i}$ is mean-square convergent with order $1$, i.e., there exists a constant $C>0$ such that 
\[
\left(\mathbb{E}\left[ \left\|X_{t_{i}} - \tilde X_{t_{i}}\right\|^2 \right]\right)^{\frac{1}{2}} \leq C\Delta,
\]
for all time points $t_i$, $i=1,\ldots,i_{\text{max}}$, where $\|\cdot \|$ denotes the Euclidean norm. 
\end{theorem}
\begin{proof}[Proof of Theorem \ref{thm:convergence}]
Let us denote by $\tilde{X}^{EM}$ a numerical solution of SDE \eqref{eq:hawkes_approximation} obtained via the Euler-Maruyama method, that is
\begin{equation}\label{eq:EM_scheme}
\tilde X^{EM}_{t_{i+1}} = \tilde X^{EM}_{t_i} + \Delta\left( A\tilde X^{EM}_{t_i}+B(\tilde X^{EM}_{t_i}) \right) + \frac{\sqrt{\Delta}}{\sqrt N} \sigma(\tilde X^{EM}_{t_i})\xi_i.
\end{equation}
First, we show that the Euler-Maruyama method, when applied to system \eqref{eq:hawkes_approximation}, coincides with the Milstein scheme, which is known to converge with mean-square order $1$. To do so, we denote the vector $x$ by $x=(x^1,\ldots,x^{\kappa})$, where $\kappa=\eta_1+\eta_2+2$. Further, we recall that the $j$-th component, $j=1,\ldots,\kappa$, of the Milstein scheme only differs from the $j$-th component of the Euler-Maruyama scheme \eqref{eq:EM_scheme} by the following additional term
\begin{equation*}
    \sum_{m_1, m_2 = 1}^{2}\sum_{l=1}^{\kappa} \sigma^{l,m_1} \frac{\partial \sigma^{j,m_2}}{\partial x^l}I(m_1,m_2),
\end{equation*}
where $\sigma^{j,m}$ denotes the value of the element at the $j$-th row and the $m$-th column of the diffusion matrix $\sigma$ at time $t_i$ and 
\begin{equation*}
    I(m_1,m_2):=\int_{t_i}^{t_{i+1}} \int_{t_i}^{{s_2}} dW_{s_1}^{m_1} dW_{s_2}^{m_2}.
\end{equation*}
Now note that the term ${\partial \sigma^{j,m_2}}/{\partial x^l}$ is only different from $0$ for $j=\eta_1+1$, $m_2=1$, $l=\eta_1+2$ and for $j=\eta_1+\eta_2+2$, $m_1=2$, $l=1$. However, $\sigma^{l,m_1}$ equals $0$ for those values of $l$. Thus, the above double sum equals $0$ and the Euler-Maruyama method coincides with the Milstein scheme.
This implies that 
\begin{equation}\label{proof:conv_EM}
     \| X_{t_i}-\tilde{X}^{EM}_{t_i} \|_{L^2} \leq C \Delta,
\end{equation}
where $\|\cdot \|_{L^2}:=\left( \mathbb{E}[ \|\cdot \|^2 ] \right)^{1/2}$ denotes the $L^2$-norm and $C$ is some generic constant.
For the second part, we provide a proof similar to the one presented in \cite{Milstein2003}. Applying the triangle inequality yields that
\begin{equation*}
    \| X_{t_i}-\tilde{X}_{t_i} \|_{L^2} \leq  \| X_{t_i}-\tilde{X}^{EM}_{t_i} \|_{L^2} + \| \tilde{X}^{EM}_{t_i}-\tilde{X}_{t_i} \|_{L^2}.
\end{equation*}
Given $X_{t_i}:=x$, let us denote with $\tilde{X}^{EM}_{t_{i+1}}(x,t_i)$ and $\tilde{X}_{t_{i+1}}(x,t_i)$ the one-step approximation of the Euler-Maruyama and splitting scheme, respectively. For instance, $\tilde{X}^{EM}_{t_{i+1}}(x,t_i)$ is given by Equation \eqref{eq:EM_scheme} with $\tilde X^{EM}_{t_i}$ replaced by $x$.
By the definition of the matrix exponent, i.e., $e^{A\Delta} := I + \Delta A + \frac{\Delta^2}{2} A^2 + O(\Delta^3)$, and by recalling \eqref{eq:split_scheme_LT2}, we obtain that

\begin{align*}
     \tilde X^{EM}_{t_{i+1}}(x,t_i)- \tilde X_{t_{i+1}}(x,t_i) &= x  + \Delta A x +   \Delta  B( x) + \frac{\sqrt{\Delta}}{\sqrt N} \sigma(x)\xi_i \\ &-  e^{A\Delta}\left( x + \Delta  B( x) + \frac{\sqrt{\Delta}}{\sqrt N} \sigma(x)\xi_i \right) \\ &= x  + \Delta Ax + \Delta  B( x) + \frac{\sqrt{\Delta}}{\sqrt N} \sigma(x)\xi_i \\ &-   x - \Delta  B( x) - \frac{\sqrt{\Delta}}{\sqrt N} \sigma(x)\xi_i \\ &-
    \Delta Ax - \Delta^2 A B( x) -  \frac{\Delta^{\frac{3}{2}}}{\sqrt N} \sigma(x)\xi_i + O(\Delta^3) \\ &=  - \Delta^2 A B( x) -  \frac{\Delta^{\frac{3}{2}}}{\sqrt N} \sigma(x)\xi_i + O(\Delta^3)
\end{align*}
Consequently, we get that
\begin{align*}
  \left\| \mathbb{E}\left[\tilde X^{EM}_{t_{i+1}}(x,t_i)- \tilde X_{t_{i+1}}(x,t_i) \right] \right\|=O(\Delta^2) , \\ \quad \left\|\tilde X^{EM}_{t_{i+1}}(x,t_i)- \tilde X_{t_{i+1}}(x,t_i)  \right\|_{L^2} = O( \Delta^{\frac{3}{2}}).
\end{align*}
Let us mention that the two bounds above do not depend on $x$ because $B$ and $\sigma$ are uniformly bounded.
Recalling \eqref{proof:conv_EM}, the result follows from the fundamental theorem on the mean-square order of convergence, see Theorem 1.1. in \cite{Milstein2004}.
\end{proof}
Theorem \ref{thm:convergence} states that as $\Delta\to 0$, the approximated solution $(\tilde X_{t_i})_{i=0,\ldots,i_{\text{max}}}$ converges to the true process $(X_t)_{t\geq 0}$ in the mean-square sense. 
In practice, however, fixed time steps $\Delta>0$ are required. Thus, there is not yet any guarantee that the constructed numerical solutions share the same properties as the true solution of \eqref{eq:hawkes_approximation}. 
For these reasons, in addition, we study the ability of $(\tilde X_{t_i})_{i=0,\ldots,i_{\text{max}}}$ to preserve the properties of SDE \eqref{eq:hawkes_approximation}. 

Note also that, different to ODE systems \citep{Hairer2006}, for stochastic equations the theoretical order of convergence usually cannot be increased by using the Strang composition instead of the Lie-Trotter approach. In practice, however, the Strang splitting often performs better than the Lie-Trotter method, see, e.g., \cite{Ableidinger2017,Buckwar2019}. This is also confirmed by our numerical experiments in Section \ref{sec:simulation_study}.

\subsection{Moment bounds of the approximated process} 
We are now interested in studying the qualitative properties of the splitting schemes for fixed time steps $\Delta>0$. We start by illustrating that the constructed splitting schemes preserve the convolution-based structure of the model derived in Proposition \ref{prop:convolution}. Using the one-step approximation \eqref{eq:split_scheme_LT2} and performing back iteration yields
\begin{equation}\label{eq:conv_splitting}
\tilde X_{t_i} = e^{At_i} x_0 + \Delta  \sum_{l=1}^{i} e^{At_l}B(\tilde X_{t_{i-l}}) + \frac{\sqrt{\Delta}}{\sqrt N}\sum_{l=1}^{i} e^{At_l} \sigma(\tilde X_{t_{i-l}})\xi_{i-l}.
\end{equation}
Note that the first term on the right side of \eqref{eq:Conv_X} is preserved exactly. Moreover, the sums in \eqref{eq:conv_splitting} correspond to approximations of the integrals in \eqref{eq:Conv_X} using the left point rectangle rule. 
Expression \eqref{eq:conv_splitting} allows to derive moment bounds for the numerical process in a similar fashion as presented for the continuous process in the previous section.
\begin{theorem}[First moment bounds of the approximated process]\label{thm:moment_bounds_scheme} 
Grant assumption \ref{assumption:a}. 
The following bounds hold for the components of $\mathbb{E}[\tilde X_{t_i}]$:
\begin{equation*}
  \tilde{\mathcal{I}}^{k,j}_{\min} \leq  \mathbb{E}[\tilde X^{k,j}_{t_i}] \leq \tilde{\mathcal{I}}^{k,j}_{\max},
\end{equation*}
where 
\begin{align*}
    \tilde{\mathcal{I}}^{k,j}_{\min} &= \left(e^{A{t_i}} x_0 \right)^{k,j} + \Delta \sum_{l=0}^{i} 
    e^{-{\nu_k}t_l} t_l^{\eta_k+1-j}\min\left\{0, \frac{c_k f_{k+1}^{\max}}{(\eta_k+1-j)! } \right\} \\ 
    \tilde{\mathcal{I}}^{k,j}_{\max} &= \left(e^{At_i}  x_0\right)^{k,j} + \Delta \sum_{l=0}^{i} 
    e^{-{\nu_k}t_l} t_l^{\eta_k+1-j}\max\left\{0, \frac{c_k f_{k+1}^{\max}}{(\eta_k+1-j)! } \right\}.
\end{align*}
\end{theorem}
\begin{proof}[Proof of Theorem \ref{thm:moment_bounds_scheme}]
From Remark \ref{rem:expA} and  \eqref{eq:conv_splitting}, it follows that
\begin{align*}\label{rem:Conv_tilde_X_1}
\tilde{X}^{k}_{t_i} = (e^{At_i}x_0)^k+\Delta \sum_{l=1}^{i} e^{A_{\nu_k}t_l} B^k(\tilde{X}^{k+1}_{t_{i-l}}) + \frac{\sqrt{\Delta}}{\sqrt{N}} \sum_{l=1}^{i} e^{A_{\nu_k}t_l} \sigma^k(\tilde{X}^{k+1}_{t_{i-l}}) \xi_{i-l}.
\end{align*}
Consequently, the $j$-th components are given by
\begin{align*}
\tilde{X}^{k,j}_{t_i} &= \underbrace{(e^{At_i}x_0)^{k,j}}_{=T_1(t_i)}+\underbrace{ \frac{1}{(\eta_k+1-j)!} \Delta \sum_{l=1}^{i}  c_k f_{k+1}(\tilde{X}^{{k+1},1}_{t_{i-l}}) e^{-{\nu_k}t_l}  t_l^{\eta_k+1-j}}_{:=\tilde{T}_2(t)} \\ &+ \underbrace{\frac{1}{(\eta_k+1-j)!} \frac{\sqrt{\Delta}}{\sqrt{N}}  \sum_{l=1}^{i} \frac{c_k}{\sqrt{p_{k+1}}}\sqrt{f_{k+1}(\tilde{X}^{{k+1},1}_{t_{i-l}})} e^{-{\nu_k}t_l}  t_l^{\eta_k+1-j} \xi^{k+1}_{i-l}}_{:=\tilde{T}_3(t_i)}.
\end{align*}
Note that, $\mathbb{E}[T_1(t_i)]=T_1(t_i)$,  and $\mathbb{E}[\tilde{T}_3(t)]=0$. The fact that the intensity function is bounded by $0<f_{k+1}\leq f_{k+1}^{\max}$  implies the result.  
\end{proof}
Note that the bounds obtained in Theorem \ref{thm:moment_bounds_scheme} equal those derived in Theorem \ref{thm:moment_bounds_process}, up to replacing the integrals (calculated in the proof of Theorem \ref{thm:moment_bounds_process}) by left Riemann sums.  The accuracy of this approximation depends on the step size $\Delta$. Under reasonably small choices of $\Delta$, the bounds are preserved accurately for all $t_i$. This is illustrated in the left panel of Figure \ref{fig:bounds_continuous_and_approximated}, where we plot the  first moment bound of the process (main variable of an excitatory population) and the one of the approximated process, derived in Theorem \ref{thm:moment_bounds_process} and Theorem \ref{thm:moment_bounds_scheme}, respectively. Different choices of $\nu_k$ are compared and for the bound of the approximated process $\Delta=0.1$ is used.

The following Corollary gives an intuition of the long-time behaviour of the bounds.
\begin{corr}\label{cor:asymptotic_bounds_splitting}
(i) The following bounds hold for the components of $\mathbb{E}[\tilde{X}_{t_i}]$ as $i \to \infty$ (and $\Delta$ fixed):
\begin{multline*}
\Delta^{\kappa^{k,j}+1}  Li_{-\kappa^{k,j}}\left(e^{-\nu_k\Delta} \right)\min\left\{0, \frac{f_{k+1}^{max}c_k }{\kappa^{k,j}!} \right\} \leq \lim_{i\to\infty} \mathbb{E}[\tilde{X}^{k,j}_{t_i}] \\ \leq \Delta^{\kappa^{k,j}+1} Li_{-\kappa^{k,j}}\left(e^{-\nu_k\Delta} \right)\max\left\{0,\frac{f_{k+1}^{max}c_k}{\kappa^{k,j}!} \right\},
\end{multline*}
where $\kappa^{k,j} := \eta_k+1-j$ and $Li_{-\kappa^{k,j}}\left(e^{-\nu_k\Delta} \right)$ is a polylogarithm function, which can be written as
\begin{equation*}
    Li_{-\kappa^{k,j}}\left(e^{-\nu_k\Delta} \right) = \left(-1 \right)^{\kappa^{k,j}+1} \sum_{l=0}^{\kappa^{k,j}}l!\: S(\kappa^{k,j}+1 ,l+1)\left(\frac{-1}{1-e^{-\nu_k\Delta}}  \right)^{l+1},
\end{equation*}
where $S(\kappa^{k,j}+1 ,l+1)$ denotes the Stirling numbers of second kind (\cite{Rennie1969}). 

(ii) The following bounds hold for the components of $\mathbb{E}[\tilde{X}_{t_i}]$ as $i \to \infty$ and $\Delta \to 0$:
\begin{equation*}
 \min\left\{0, \frac{c_kf_{k+1}^{max}}{\nu^{\kappa^{k,j}+1}_k}\right\} \leq \lim_{\Delta\to 0} \lim_{i\to\infty} \mathbb{E}[\tilde{X}^{k,j}_{t_i}] \leq \max \left\{0, \frac{c_kf_{k+1}^{max}}{\nu^{\kappa^{k,j}+1}_k}\right\}.
\end{equation*}
\end{corr}

\begin{proof}
\textbf{(i)} The zero bound is trivial.  Considering
\begin{align*}
    \lim_{i\to\infty}\sum_{l=0}^{i}e^{-{\nu_k}t_l} t_l^{\kappa^{k,j}} &= \lim_{i\to\infty}\sum_{l=0}^{i}e^{-{\nu_k}l\Delta } (l\Delta)^{\kappa^{k,j}} \\ &=  \lim_{i\to\infty}\Delta^{\kappa^{k,j}}\sum_{l=0}^{i}e^{-{\nu_k}l\Delta} l^{\kappa^{k,j}} = \Delta^{\kappa^{k,j}} Li_{-\kappa^{k,j}}\left(e^{-\nu_k\Delta} \right)
\end{align*}
gives the result. The explicit form of the function is given in \cite{Wood1992}.

\textbf{(ii)} Let us rewrite once again the expression included in the limit:
\begin{eqnarray*}
   &\left(-1 \right)^{\kappa^{k,j}+1}& \lim_{\Delta\to 0} \frac{(\nu_k\Delta)^{\kappa^{k,j}+1}}{\nu^{\kappa^{k,j}+1}_k\kappa^{k,j}!} \sum_{m=0}^{\kappa^{k,j}} m!\: S(\kappa^{k,j}+1 ,m+1)\left(\frac{-1}{1-e^{-\nu_k\Delta}} \right)^{m+1} \\  &=& \lim_{\Delta\to 0} \left[\frac{S(\kappa^{k,j}+1 ,\kappa^{k,j}+1)\kappa^{k,j}!}{\nu^{\kappa^{k,j}+1}_k\kappa^{k,j}!} \left(\frac{\nu_k\Delta}{1-e^{-\nu_k\Delta}} \right)^{\kappa^{k,j}+1} \right. \\  &-&  \left. \Delta \frac{S(\kappa^{k,j}+1 ,\kappa^{k,j})(\kappa^{k,j}-1)!}{\nu^{\kappa^{k,j}}_k\kappa^{k,j}!} \left(\frac{\nu_k\Delta}{1-e^{-\nu_k\Delta}} \right)^{\kappa^{k,j}} + O(\Delta^2) \right].
\end{eqnarray*}
Note that $\lim_{\Delta\to 0} \left(\frac{\nu_k\Delta}{1-e^{-\nu_k\Delta}} \right) = 1$. This implies that in the limit ${1}/{\nu^{\kappa^{k,j}+1}_k}$ is the only remaining term, since the rest converges to $0$ as $\Delta\to 0$. This gives the result.
\end{proof}
\begin{figure}
\begin{centering}
    \includegraphics[width=0.48\textwidth]{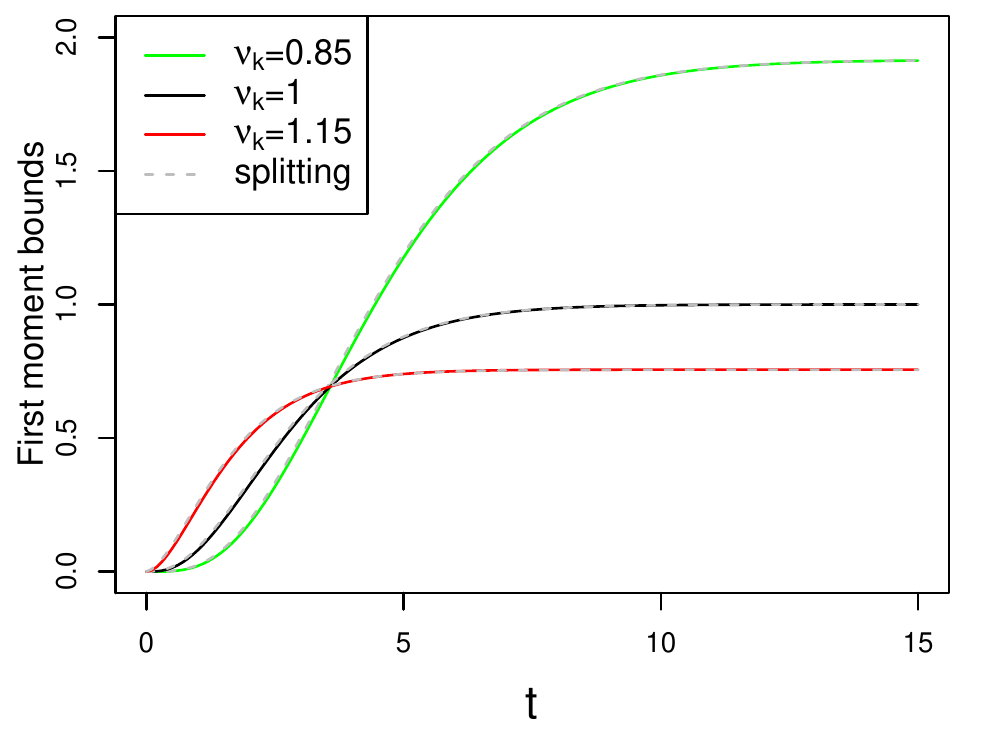}	
    \includegraphics[width=0.48\textwidth]{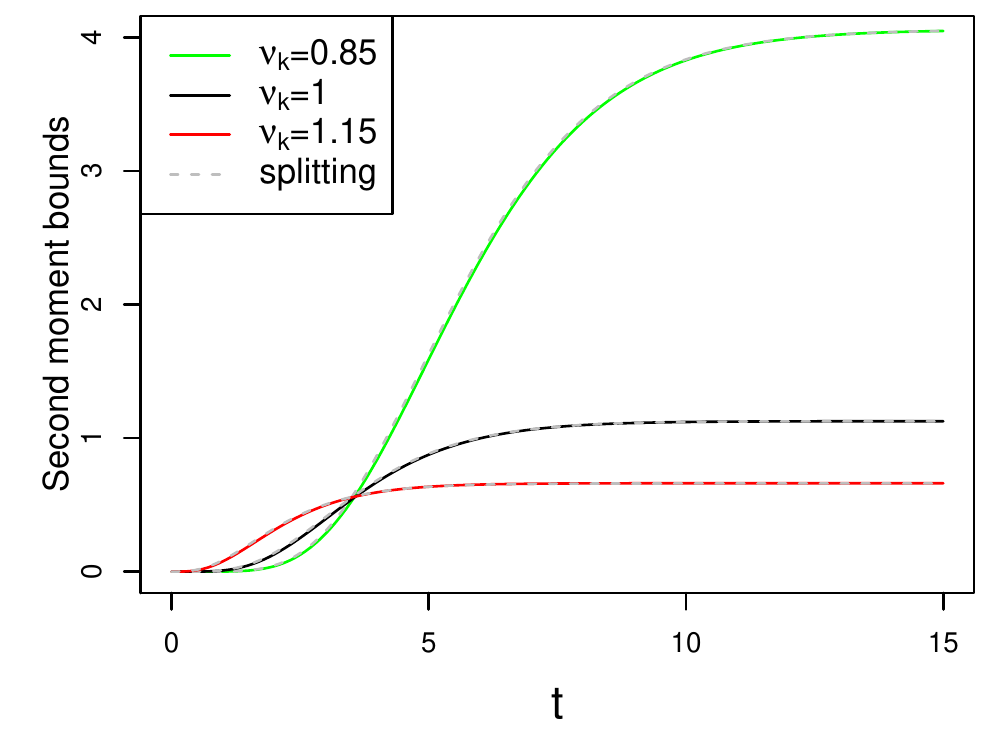}	
	\caption{
	First (left panel) and second (right panel) moment bounds of the excitatory population $k=2$ for different values of $\nu_2$. The moment bounds for the diffusion are in solid lines and the moment bounds for the splitting scheme are in dashed lines. The bound of the rate function is fixed to $f_1^{\max} = 1$. The parameters are $\eta_2=3$, $N=100$, $p_{1}=1/2$ and the time step $\Delta=0.1$ is used.
	}
	\label{fig:bounds_continuous_and_approximated}
\end{centering}
\end{figure}
In the first part of Corollary \ref{cor:asymptotic_bounds_splitting}, the sums in Theorem \ref{thm:moment_bounds_scheme} are calculated explicitly as $i \to \infty$. This limit is described by polylogarithm functions. Note that the zero bounds derived in Remark \ref{thm:moment_bounds_process}, i.e., the upper bounds for the inhibitory population ($c_k = -1$) and the lower bounds for the excitatory population ($c_k = +1$) are preserved exactly by the splitting scheme for all times $t_i$ and for all choices of $\Delta>0$. Moreover, the lower bounds for the inhibitory population and the upper bounds for the excitatory population are preserved accurately as $i \to \infty$, provided that $\Delta$ is reasonably small. Indeed,  
as $i\to\infty$ and $\Delta\to 0$ (second part of Corollary \ref{cor:asymptotic_bounds_splitting}), the bounds coincide with the ones obtained in Remark \ref{rem:bounds_process}. 
\begin{theorem}[Second moment bounds of the approximated process]\label{thm:second_moment_splitting}
Grant assumption \ref{assumption:a}. 
Each component of $\mathbb{E}[(\tilde X^k_t)^2]$ is bounded by
\begin{align*}
    \mathbb{E}[(\tilde X^{k,j}_t)^2] &\leq \left( \left(e^{At}  \ x_0\right)^{k,j} \right)^2 + 2  \left(e^{At} x_0\right)^{k,j}\max\left\{0,\frac{c_k f^{\max}_{k+1}}{(\eta_k+1-j)!} \tilde I^{k,j}_1(t)\right\} \\ &+ 
    f^{\max}_{k+1}\left(\frac{c_k }{(\eta_k+1-j)!}\right)^2\left(\sqrt{f^{\max}_{k+1}} \: \tilde I^{k,j}_1(t) + \sqrt{\frac{\tilde I^{k,j}_2(t)}{N\cdot p_{k+1}}}\right)^2,
\end{align*}
where $\tilde I^{k,j}_u(t), \: u = 1,2$, are defined as
\begin{equation*}
    \tilde I^{k,j}_u(t) := \Delta\sum_{l=0}^i e^{-u\nu_kt_l}t_l^{u(\eta_k+1-j)}.
\end{equation*}
\end{theorem}
\begin{proof}[Proof of Theorem \ref{thm:second_moment_splitting}]
The proof repeats the proof of Theorem \ref{thm:second_moment_bounds_process}, up to replacing integrals $I^{k,j}_u(t)$ by sums $\tilde I^{k,j}_u(t)$. 
\end{proof}
Similar to before, the second moment bounds obtained for the splitting scheme equal those derived for the true process in Theorem \ref{thm:second_moment_bounds_process}, except that the integrals are replaced by corresponding Riemann sums. 
Using the same arguments as in the proof of Corollary \ref{cor:asymptotic_bounds_splitting}, we conclude that also the second moment bounds are preserved accurately by the splitting scheme for reasonable choices of the time step $\Delta$. A comparison of the theoretical and discrete second moment bounds is provided in the right panel of Figure \ref{fig:bounds_continuous_and_approximated}. 

\subsection{Geometric ergodicity of the approximated process}
Finally, our aim is to prove that the splitting scheme preserves the ergodic property of the underlying process in the spirit of \cite{Mattingly2002, Ableidinger2017}, providing a discrete analogue of  Proposition \ref{prop:ergodicity_process}. 
The main step is to establish a discrete Lyapunov condition for the approximated solution $(\tilde X_{t_i})_{i=0,\ldots,i_{\text{max}}}$. It is granted by the following lemma. 
\begin{lemma}[Lyapunov condition for the approximated process] \label{lemma:ergodicity}
Grant assumption \ref{assumption:a}. The functional $\tilde G$, given by 
\[
\tilde G(x) = \sum_{k=1}^2\sum_{j=1}^{\eta_k+1}\frac{j}{\nu_k^{j-1}}\left|x^{k,j}\right|,
\]
is a Lyapunov function for $\tilde X$, i.e., there exist constants $\alpha \in [0,1)$ and $\beta \geq 0$, such that
\[
\mathbb{E}\left[\tilde G(\tilde X_{t_{i+1}})\vert \tilde X_{t_i} \right] \leq \alpha \tilde  G(\tilde X_{t_i})+\beta.
\]
\end{lemma}
\begin{proof}
We bound the approximated solution obtained via \eqref{eq:split_scheme_LT2} from above by a sum of three terms, thanks to the triangle inequality:
 \begin{align*}
     \tilde G( \tilde X_{t_{i+1}}) &= \tilde G\left(e^{A\Delta} \tilde X_{t_i} + \Delta e^{A\Delta} B( \tilde X_{t_i}) + \frac{\sqrt{\Delta}}{\sqrt N} e^{A\Delta} \sigma(  \tilde X_{t_i})\xi_i\right)  \\ &\leq
     \underbrace{\tilde G\left(e^{A\Delta} \tilde X_{t_i}\right)}_{T_1} + \underbrace{\Delta \tilde G\left(e^{A\Delta} B( \tilde X_{t_i})\right)}_{T_2} + \underbrace{\frac{\sqrt{\Delta}}{\sqrt N} \tilde G\left(e^{A\Delta} \sigma( \tilde X_{t_i})\xi_i\right)}_{T_3}.
 \end{align*}
Note that the term $T_2$, as well as the expectation of $T_3$ is bounded by a constant depending on $f^{\max}_k$, so that $\mathbb{E}[T_2\vert \tilde X_{t_i}]+\mathbb{E}[T_3\vert \tilde X_{t_i}]\leq \beta$, and $\beta > 0$ since we consider the absolute value.
Further, using the formulas \eqref{eq:A_exponential}-\eqref{eq:jth_element}, we can expand $T_1$ as follows.
\begin{eqnarray*}
    &&\tilde G\left(e^{A\Delta} \tilde X_{t_i}\right) =  \sum_{k=1}^2e^{-\nu_k\Delta}\sum_{j=1}^{\eta_k+1}\frac{j}{\nu_k^{j-1}}\left|\sum_{m=j}^{\eta_k+1} \frac{\Delta^{m-j}}{(m-j)!} \tilde X^{k,m}_{t_i}\right|\\ &\leq& \sum_{k=1}^2e^{-\nu_k\Delta}\sum_{j=1}^{\eta_k+1}\frac{j}{\nu_k^{j-1}}\sum_{m=j}^{\eta_k+1} \frac{\Delta^{m-j}}{(m-j)!}\left| \tilde X^{k,m}_{t_i}\right| \\ &=& \sum_{k=1}^2e^{-\nu_k\Delta}\left[\sum_{j=1}^{\eta_k+1}\frac{j}{\nu_k^{j-1}} \left| \tilde X^{k,j}_{t_i}\right|+\Delta \sum_{j=2}^{\eta_k+1}\frac{j-1}{\nu_k^{j-2}} \left| \tilde X^{k,j}_{t_i}\right|+\dots + \frac{\Delta^{\eta_k}}{\eta_k!} \left| \tilde X^{k,\eta_k+1}_{t_i}\right|\right]  \\ &=& \sum_{k=1}^2e^{-\nu_k\Delta}\left[\sum_{j=1}^{\eta_k+1}\frac{j}{\nu_k^{j-1}} \left| \tilde X^{k,j}_{t_i}\right|+\Delta\nu_k \sum_{j=2}^{\eta_k+1}\frac{j-1}{\nu_k^{j-1}} \left| \tilde X^{k,j}_{t_i}\right|+\dots + \frac{(\nu_k\Delta)^{\eta_k}}{\eta_k!} \frac{1}{\nu_k^{\eta_k}}\left| \tilde X^{k,\eta_k+1}_{t_i}\right|\right] .
\end{eqnarray*}
Note that, since $\nu_k>0$, for all $m\geq 1$ it holds that
\[
\sum_{j=m}^{\eta_k+1}\frac{(j-m+1)}{\nu_k^j}\left| \tilde X^{k,j}_{t_i}\right| \leq  \sum_{j=1}^{\eta_k+1}\frac{j}{\nu_k^j}\left| \tilde X^{k,j}_{t_i}\right| = \tilde G\left( \tilde X^k_{t_i}\right).
\]
Thus, we have 
\begin{equation*}
  \tilde  G\left(e^{A\Delta} \tilde X_{t_i}\right)  \leq \sum_{k=1}^2\left(e^{-\nu_k\Delta}\sum_{j=0}^{\eta_k}\frac{(\nu_k\Delta)^{j}}{j!}  \right) \tilde G\left(  \tilde X^k_{t_i}\right) .
\end{equation*}
 Denote $\alpha = \max_k \left(e^{-\nu_k\Delta}\sum_{j=0}^{\eta_k}\frac{(\nu_k\Delta)^{j}}{j!}  \right)$. Since $\eta_k$ is finite, we get $\alpha < 1$, which implies  the result.
\end{proof}
Note that the statement of Lemma \ref{lemma:ergodicity} holds without any assumption on the time step $\Delta$. Also, the Lyapunov function is the same as for the continuous process up to smoothing the absolute value (see \eqref{eq:lyapunov_eva}).  
Having established a discrete Lyapunov condition, the ergodicity is conditioned on two further technical steps. First, the transition probability of two (or more) consecutive steps, given by the recursive relation \eqref{eq:split_scheme_LT2}, must have a smooth transition density. This fact is granted by the hypoellipticity of the scheme (see Remark \ref{rem:hypoellipticity_scheme}).

Second, the irreducibility condition must hold. It means that any point $y\in \mathbb{R}^\kappa$ could be reached from any starting point $x\in \mathbb{R}^\kappa$ in a fixed number of steps. In other words, we need a discrete-time analogue of Theorem 4 in \cite{Locherbach2019}, granting the controllability of SDE \eqref{eq:hawkes_approximation}. It is the following Lemma, which is proved in Appendix \ref{app:control}.
\begin{lemma}[Irreducibility condition]\label{lemma:irreducibility}
Grant assumption \ref{assumption:a}. Denote $\eta^* = \max_k\{\eta_k\}$. Then, for all $ x, y \in \mathbb{R}^\kappa$ there exists some sequence of 2-dimensional vectors $(\xi_i)_{i=1,\dots,\eta^*+1}$ such that
\begin{equation}\label{eq:irreducibility} 
y = \underbrace{\left(\psi_\Delta[\xi_{\eta^*+1}] \circ \dots \circ \psi_\Delta[\xi_{1}] \right)}_{\eta^*+1} (x),  
 \end{equation}
 where $\psi_\Delta$ denotes one step of the scheme defined by \eqref{eq:split_scheme_LT2}, where the notation $[\cdot]$ is introduced to stress the dependency on the vectors $(\xi_i)_{i=1,\dots,\eta^*+1}$.
\end{lemma}

Lemmas \ref{lemma:ergodicity} and \ref{lemma:irreducibility}, combined with the hypoellipticity of the scheme gives the following result, which is analogous to Theorem 7.3 in \cite{Mattingly2002}. 
\begin{theorem}[Geometric ergodicity]
Grant Assumption \ref{assumption:a}. Then the process $(\tilde X_{t_i})_{i=0,\dots,i_{\max}}$ has a unique invariant measure $\pi^\Delta$ on $\mathbb{R}^\kappa$. For all initial conditions $x_0$ and all $m\geq 1$, there exist $\tilde C = C(m, \Delta) >0$ and  $\tilde\lambda = \tilde\lambda(m,\Delta)>0$ such that, for all measurable functions $g:\mathbb{R}^\kappa\to \mathbb{R}$ such that $|g|\leq \tilde G^m$, 
\begin{equation*}
    \forall i=0,\dots,i_{\max}, \quad \left| \mathbb{E} g(\tilde X_{t_i}) - \pi^\Delta(g)\right|\leq \tilde C \tilde G(x_0)^m e^{-\tilde \lambda t_i}.
\end{equation*}
\end{theorem}

\section{Thinning procedure for the simulation of the PDMP}\label{section:simulation_PDMP}

In this section we explain the simulation method for the multidimensional point process characterized by the intensities \eqref{eq:intensity}. This part is motivated by the fact that, on the contrary to the diffusion, the simulation of the PDMP can be exact. By that, we mean that the result of the simulation is a realization of $(\bar{X}_t)_{t\geq 0}$. In comparison, the result of the simulation of the diffusion $(X_t)_{t\geq 0}$ is in fact the discrete time process $(\tilde{X}_{t_i})_{i=0,\dots,i_{\max}}$.
This allows us to compare the PDMP \eqref{eq:PDMP_flow} with the stochastic diffusion defined through \eqref{eq:hawkes_approximation}, which we treat via the property-preserving splitting scheme, introduced in the previous section. 

We choose the thinning procedure which dates back to \cite{lewis1979sim} and \cite{ogata1981on}. 
It is based on the rejection principle and relies on the following fact. In order to simulate a point process $Z$ according to the stochastic intensity $\lambda_t$, it is sufficient to simulate some (dominating) point process $\tilde{Z}$ with (dominating) predictable piece-wise constant intensity $\tilde{\lambda}$ such that $\lambda_t \leq \tilde{\lambda}_t$. During the simulation of $\tilde{Z}$, each new simulated spiking time $\tilde{T}$ for $\tilde{Z}$ is kept as a point of $Z$ with probability ${\lambda}_{\tilde{T}}/\tilde\lambda_{\tilde{T}}$ (independently from every other point). Otherwise, $\tilde{T}$ is discarded. 
The efficiency of the thinning procedure is highly related to the sharpness of the upper-bound $\tilde{\lambda}$. The sharper the bound, the less rejections are made and the more efficient is the procedure.

Note that the case $\eta_k=0$ corresponds to the exponential kernel. The simulation of Hawkes processes with an exponential kernel is widely studied and there exist several implemented packages, e.g., for the software \textsf{R}. Moreover, apart from the thinning procedure, other exact simulation algorithms are available, see, in particular, \cite{dassios2013exact}. To the best of our knowledge, the only reference for the case when $\eta_k \geq 1$ is \cite{Duarte2019}. The aim of the current section is to generalize the algorithm presented in the above mentioned work to the case of multiple populations and to provide a more efficient upper bound $\tilde{\lambda}$. In particular, our approach allows for an efficient handling of rapidly increasing intensity functions.

\subsection{Choice of an upper bound for the intensity}
If $\bar{Z}_t= 0$, i.e., in absence of any spike, it follows from \eqref{eq:PDMP_flow} that $\bar{X}$ evolves as a linear ODE with matrix $A$ so that $\bar{X}_t = e^{At}x_0$. In particular, for all neurons $n=1,\dots, N_k$, it follows that 
\begin{equation}\label{eq:intensity:no:spike}
    \lambda^{k,n}_t = f_k((e^{At}x_0)^{k,1}).
\end{equation}
One possible choice for the dominating intensity $\tilde{\lambda}$ in the thinning procedure is to provide an upper-bound of \eqref{eq:intensity:no:spike} which holds for all $t\geq 0$. 
A straightforward candidate for such a bound is provided in the following lemma.
\begin{lemma}\label{lem:bound:intensity}
For any $x\in \mathbb{R}^\kappa$, let $\Phi_k(x) = \sup_{t\geq 0} (e^{At}x)^{k,1}$. Then,
\begin{equation}
    \Phi_k(x) \leq \tilde{\Phi}_k(x) = \max_{j=1,\dots, \eta_k +1} \left\{ 0, \frac{x^{k,j}}{\nu_k^{j-1}} \right\}.
\end{equation}
\end{lemma}
\begin{proof}
The explicit expression of $(e^{At}x)^{k,1}$ is given in \eqref{eq:jth_element}, that is:
\begin{equation*}
    (e^{At}x)^{k,1} = e^{-\nu_k t} \left( x^{k,1} + tx^{k,2} +\dots + \frac{t^{\eta_k}}{\eta_k !} x^{k,\eta_k+1} \right).
\end{equation*}
Setting $y_j = x^{k,j}/(\nu_k)^{j-1}$, one gets 
\begin{equation}
    (e^{At}x)^{k,1} = e^{-\nu_k t} \left( y_1 + t\nu_k y_2 +\dots + \frac{(t\nu_k)^{\eta_k}}{\eta_k !} y_{\eta_k+1} \right) \leq \max_k\{0,y_k\} e^{-\nu_k t} g(t).
\end{equation}
The result follows from the fact that $g(t) = 1+t\nu_k + \dots + (t\nu_k)^{\eta_k}/\eta_k! \leq e^{\nu_k t}$.
\end{proof}

\begin{remark}
Another possible choice of a uniform bound, similar to the one given in Lemma \ref{lem:bound:intensity}, is provided in \cite{Duarte2019}. Their method, adapted to our case, gives
\begin{equation}
    \Phi_k(x) \leq e \max\left\{ 1, \left(\frac{\eta_k}{e\nu_k}\right)^{\eta_k} \right\} \max_{j}\{x^{k,j}\},
\end{equation}
which is larger, and thus less efficient than the bound proposed in Lemma \ref{lem:bound:intensity}.
\end{remark}
Since the functions $f_k$ are non-decreasing, the upper-bound of $(e^{At}x)^{k,1}$ given in Lemma \ref{lem:bound:intensity} provides the bound $\tilde{f}_k(x) = f_k( \tilde{\Phi}_k(x))$ on the intensity.
However, there is no guarantee that this bound  is sharp.
In most practical cases (especially when the functions $f_k$ are increasing fast), the procedure rejects a vast majority of the simulated points. 
Hence, a more efficient approach, based on the computation of the critical points of the function $(e^{At}x)^{k,1}$, is proposed. Further, instead of considering a bound for any $t>0$ we choose a fixed time step $\tilde\Delta>0$ (such that one spike is likely to occur in the interval $[0,\tilde\Delta]$) and compute $\Phi_k^{\tilde\Delta}(x) = \sup_{0\leq t\leq  {\tilde\Delta}} (e^{At}x)^{k,1}$ instead of $\Phi_k(x)$. This choice has no impact on the precision of the simulation. It only influences the sharpness of the bound used in the method and thus its computational efficiency.
\begin{lemma}\label{lem:bound:intensity:local}
For any $x\in \mathbb{R}^\kappa$, it holds that
\begin{equation}
    \Phi_k^{\tilde\Delta}(x) = \max_{0<t_c <{\tilde\Delta}} \{x^{k,1}, (e^{At_c}x)^{k,1}, (e^{A{\tilde\Delta}}x)^{k,1}\},
\end{equation}
where the maximum is taken over the critical points $t_c$ of $t\mapsto (e^{At}x)^{k,1}$, that are the solutions of the equation
\begin{equation}
    (-\nu_k x^{k,1} + x^{k,2}) + \dots + (-\nu_k x^{k,\eta_k} + x^{k,\eta_k+1}) \frac{(t_c)^{\eta_k-1}}{(\eta_k-1)!} + (-\nu_k x^{k,\eta_k+1}) \frac{(t_c)^{\eta_k}}{(\eta_k)!} = 0.
\end{equation}
\end{lemma}
\begin{proof}
The result follows from the computation of the time derivative of $ (e^{At}x)^{k,1}$.
\end{proof}
The critical points in Lemma \ref{lem:bound:intensity:local} are given by polynomial roots, which can be accurately computed numerically. In most practical cases, the computational cost of the polynomial roots is compensated by the efficiency gained in the rejection method.
Finally, let us define the upper-bound intensity function by 
$$\tilde{f}_k^{\tilde\Delta}(x) = f_k(\Phi_k^{\tilde\Delta}(x)).$$
Note that when the population is inhibitory ($c_k = -1$), the naive upper-bound $\tilde{f}_k$ is constant with respect to time because all the coordinates of $\bar{X}^1$ are always negative and the bound given by Lemma \ref{lem:bound:intensity} is 0. Thus, $\tilde{f}_k\equiv f_k(0)$. Of course, such a bound is not sharp in general. However, it is interesting to see how the  two upper-bounds $\tilde{f}_k$ and $\tilde{f}^{\tilde\Delta}_k$ behave for a particular realisation of the intensity process for excitatory populations.
Figure \ref{fig:intensity_bounds} gives a comparison of the paths of $\tilde{f}_2$ and $\tilde{f}^{\tilde\Delta}_2$ for the excitatory population (with $\tilde\Delta\equiv 1$). We observe that both bounds are rather precise when the potential $\bar{X}^2_t$ (and, respectively, the intensity process) is decreasing. On these intervals the differences between the three trajectories are negligible. However, the accuracy of $\tilde{f}_2$ drops drastically on the intervals where the intensity grows. In general, the higher is the amplitude of the oscillations, the less performing is the naive bound. This is particularly visible when illustrated on systems with high memory order ($\eta_k=3$ or $6$). For $\eta_k=1$ both bounds perform good, however, $\tilde{f}^{\tilde\Delta}_k$ is clearly closer to the true process. The influence of the bound ($\tilde{f}_2$ or $\tilde{f}^{\tilde\Delta}_2$) on the execution time is discussed in Section \ref{sec:simulation_study}.
\begin{figure}
\begin{centering}
    \includegraphics[width=\textwidth]{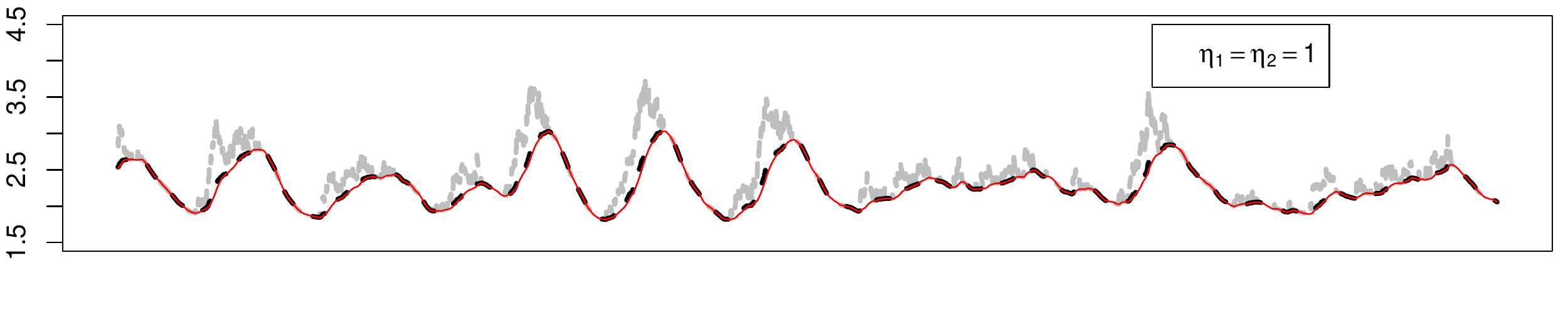}
    \includegraphics[width=\textwidth]{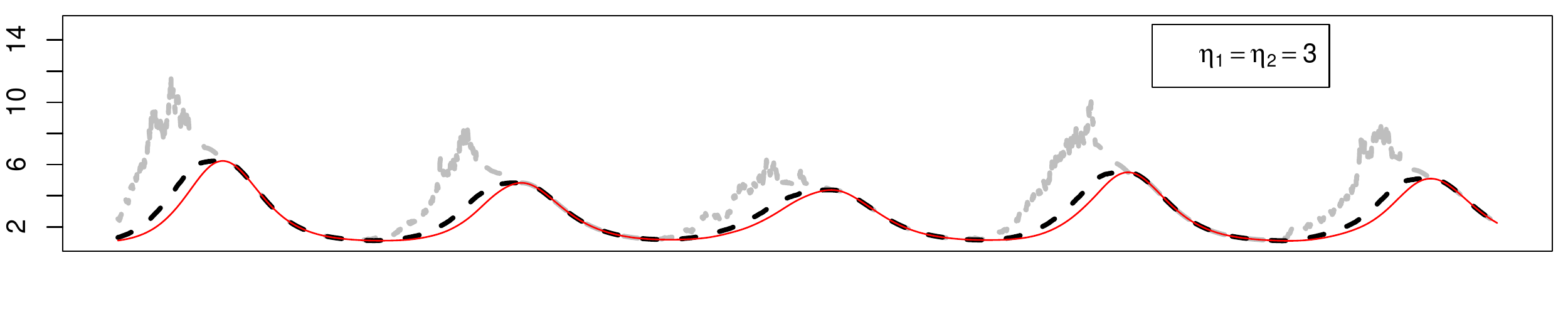}
    \includegraphics[width=\textwidth]{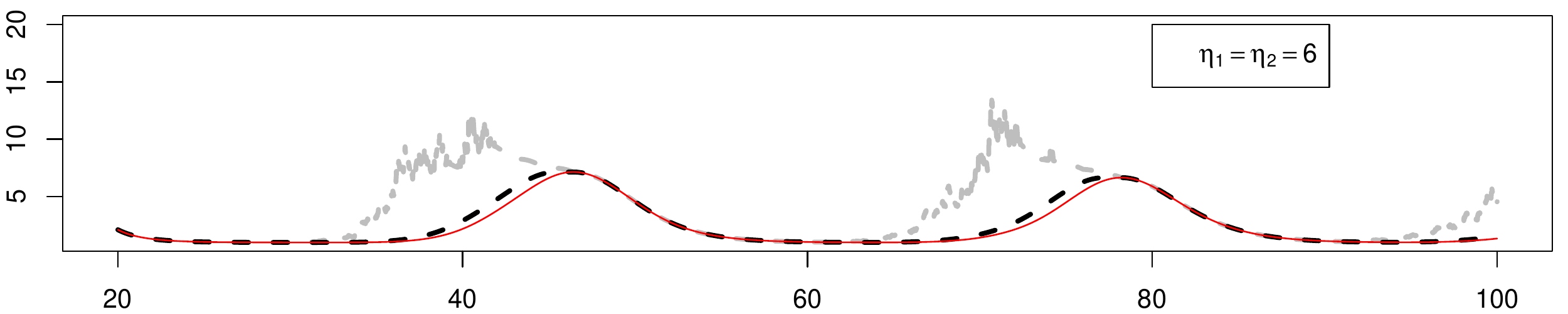}
	 \caption{Intensity and intensity bounds for the second population (excitatory) $t\in[20, 100]$. Red solid line: true intensity $\lambda^2_t$, black dash line: $\tilde{f}_2^{\tilde\Delta}(\bar X_t)$, grey dash line: $\tilde{f}_2(\bar X_t)$. Intensity functions $f_1$ and $f_2$ are given in \eqref{eq:f1:f2}, $\nu_1=\nu_2 = 0.9$, $N_1 = N_2 = 50$.}
	\label{fig:intensity_bounds}
\end{centering}
\end{figure}

\subsection{Simulation algorithm}
Now let us detail the recursive procedure, which is summarized in Algorithm \ref{algo:simulation:point:process}. We choose a discrete time step ${\tilde\Delta}$, 
a stopping time $t_{\max}$ and fix the initial values $t_0 = 0$ and $\bar{X}_0 = x_0\in \mathbb{R}^\kappa$. 
Let us assume that the procedure's current step is $i$ with current time $t_i$ and potential value $\bar{X}_i$. Let us explain how $t_{i+1}$ and $\bar{X}_{i+1}$ are obtained.
One simulates two independent exponential variables $\tau_1$ and $\tau_2$ with respective parameters $N_k\tilde{f}_k^{\tilde\Delta}(x)$ (one for each population). They represent the waiting times to the next spikes of the dominant process $\tilde{Z}$ for each respective population. Then, two cases may occur.
\begin{enumerate}
    \item If $\min\{\tau_1, \tau_2\}> {\tilde\Delta}$, no spike occurs in the interval $[t_i, t_i+{\tilde\Delta}]$. We update $t_{i+1} = t_i+{\tilde\Delta}$ and $\bar{X}_{i+1} = e^{A{\tilde\Delta}}\bar{X}_i$.
    \item  If $\tau = \min\{\tau_1, \tau_2\}\leq {\tilde\Delta}$, then the dominating point process $\tilde{Z}$ emits a spike at time $t^* = t_i+\tau$. Let us denote by $k^*$ the population with the smallest waiting time, that is $\tau = \tau_{k^*}$. It remains to decide whether $t^*$ is also a spiking time for the process $Z$. If not, this point is discarded. We draw a uniform variable $U$ on $[0,1]$ and define the threshold $R$:
    \begin{equation*}
        R := \frac{f_{k^*}\left(e^{A\tau}\bar{X}_{i} \right)}{\tilde{f}_{k^*}^{\tilde\Delta}(\bar{X}_{i})}, \quad R \in [0,1] \text{ by the definition of  }\tilde{f}_k^{\tilde\Delta}(x).
    \end{equation*}
    \begin{itemize}
        \item If $U\geq R$, then $t^*$ is discarded, i.e., no spike occurs in the interval $[t_i, t^*]$. We update $t_{i+1} = t^*$ and  $\bar X_{i+1}=e^{A\tau}\bar X_{i}$.
        \item If $U<R$, then $t^*$ is kept, i.e., we add $t^*$ to the list of one neuron of population $k^*$ chosen uniformly at random. We update $t_{i+1} = t^*$ and $\bar X_{i+1} = e^{A\tau}\bar X_{i} + \Gamma \mathds{1}(k^*)$, where $\mathds{1}(k^*) = (\mathds{1}_{k^* = 1}, \mathds{1}_{k^* = 2})^T$.
    \end{itemize}
\end{enumerate}
Finally, the execution is stopped once $t_{i}\geq t_{\max}$, i.e., once the time horizon of interest is reached.  As output the algorithm returns a list of the spiking times of each neuron and the values of the processes $\bar{X}$ and $\lambda$ at the spiking times. 
On this stage it is clear why it is important to have a sharp upper bound. The closer the threshold $R$ is to $1$, the less points are rejected. 

\begin{algorithm}\caption{Simulation of model \eqref{eq:intensity} with $K = 2$ populations. \label{algo:simulation:point:process}}
\KwIn{intensity functions $f_1$ and $f_2$; integers $N_1$, $N_2$, $\eta_1$ and $\eta_2$; real numbers $c_1$, $c_2$, $\nu_1$, $\nu_2$, ${\tilde\Delta}$ and $t_{\max}$; real vector $x_0\in \mathbb{R}^\kappa$.}
\KwOut{point processes $(Z^{k,n})_{k=1,2;\, n=1,\dots, N_k}$, Markovian cascade process $\bar{X}$ and intensity processes $(\lambda^k)_{k=1,2}$.}
\emph{Initialization:} $t\leftarrow 0$, $x\leftarrow x_0$\;
    \While{$t<t_{\max}$}{ 
    $\tilde{\lambda}_k \leftarrow \tilde{f}_k^{\tilde\Delta}(x)$ for $k=1,2$\;
    draw $\tau_k\sim \mathcal{E}(N_k\tilde{\lambda}_k)$ for $k=1,2$\;
    $\tau \leftarrow \min_k \tau_k$ and $k^* \leftarrow \arg\min_k \tau_k$\;
        \eIf{$\tau>{\tilde\Delta}$}{
        \nlset{(1)}$t\leftarrow t+{\tilde\Delta}$ and $x\leftarrow e^{A{\tilde\Delta}}x$\;}{
        $t\leftarrow t+\tau$, $x\leftarrow e^{A\tau}x$, $\lambda_{k^*} \leftarrow f_{k^*}(x)$\;
        draw $U\sim \mathcal{U}([0,1])$\;
            \eIf{$U< \lambda_{k^*}  / \tilde{\lambda}_{k^*}$}{
            \nlset{(2)}draw $n\sim \mathcal{U}(\{1,\dots, N_{k^*}\})$ and add $t$ to the list $Z^{k^*,n}$\;
            $x\leftarrow x + \Gamma\mathds{1}(k^*)$\;
            add $x$ to the list $\bar{X}$ and $\lambda_k = f_k(x)$ to the list $\lambda^k$ for $k=1,2$\;}
            {\nlset{(3)}do nothing\;}
        }
    }
\end{algorithm}

Algorithm \ref{algo:simulation:point:process} is most efficient when every iteration of the \textbf{while} loop enters condition \textbf{(2)}. Of course, that ideal case does not occur in practice.
When lowering the value of ${\tilde\Delta}$, the number of loops satisfying condition \textbf{(3)} decreases because the dominating intensity $\tilde{\lambda}$ is getting smaller. On the other hand, the number of loops fulfilling condition \textbf{(1)} increases because the exponentially distributed times have greater chances of being larger than ${\tilde\Delta}$. 
The calibration of $\tilde{\Delta}$ is a difficult problem which is not addressed here. In practice, it is observed that the execution time is not very sensitive to the value of $\tilde{\Delta}$. The main bottleneck of the thinning algorithm is the sharpness of the intensity bound. When the intensity functions are exponential, the computational time is more than halved with the bound of Lemma \ref{lem:bound:intensity:local} compared to the bound of Lemma \ref{lem:bound:intensity}. This is illustrated in the right panel of Figure \ref{fig:elapsed_time}.


\section{Numerical experiments}\label{sec:simulation_study}

A simulation study, illustrating the theoretical results discussed in the previous sections, is now provided. It consists in two steps. First, we study the performance of the proposed splitting schemes. More precisely, we compare the Lie-Trotter \eqref{eq:split_scheme_LT2} and Strang \eqref{eq:split_scheme} splitting schemes with the Euler-Maruyama approximation. We report sample paths, empirical densities and comment also on the first and second moments. This step follows the numerical study in \cite{Ableidinger2017}, and shows that the Strang splitting performs best.
Second, we compare the diffusion process (simulated with the property-preserving Strang splitting scheme) to the PDMP, varying the number of neurons $N$. 
In particular, when comparing the long-time behaviour of the processes (see Figure \ref{fig:PDMP_and_diffusion}), we show that the diffusion approximation is less and less accurate as $t\to +\infty$. It confirms the results obtained in Theorems \ref{thm:weak_error_bound} and \ref{thm:strong:approx:diffusion}.

Following the work of \cite{Ditlevsen2017eva}, throughout this section we use the following intensity functions
\begin{equation}\label{eq:f1:f2}
f_1(x) = \begin{cases} 10e^x & \text{if } x<\log(20) \\
\frac{400}{1+400e^{-2x}} & \text{if } x\geq \log(20)
\end{cases},
\, \
f_2(x) = \begin{cases} e^x & \text{if } x<\log(20) \\
\frac{40}{1+400e^{-2x}} & \text{if } x\geq \log(20)
\end{cases}.
\end{equation}
Further, we fix the parameters $c_1 = -1, c_2 = 1$ and consider $N_1 = N_2$. Unless stated differently, throughout this section the initial condition is fixed to $x_0=0_\kappa$. The parameter $p_k$ is then defined as ${N_k}/{N}$. The fact that $c_1c_2 <0$ ensures that the population shows an oscillatory behaviour, for certain parameters $\nu_k$ and $\eta_k$ (see \cite{Ditlevsen2017eva} for further details). 

\subsection{Comparison of the Euler-Maruyama method and the splitting schemes}

In this section we are interested in comparing  the performance of the splitting schemes 
with that of the frequently applied Euler-Maruyama method (EM), for varying time steps $\Delta$.
The parameter values $\nu_1 = \nu_2 = 1$, $\eta_1 = 3$,  $\eta_2 = 2$, $N_1 = N_2 = 50$ are used and the dimension of the system is thus $\kappa = 7$. Except for the density and mean-square convergence plots, we consider the time interval $[0,100]$. Unless stated otherwise, we plot the variables $X^{k,1}_t$ for $k=1,2$ in black and the remaining $\eta_1+\eta_2$ auxiliary memory variables in grey. 

\subsubsection{Illustration of the mean-square convergence order}
We start our study by comparing the convergence rates of the EM method and the Lie-Trotter \eqref{eq:split_scheme_LT2} and Strang \eqref{eq:split_scheme} splitting schemes. 
The root mean-square error, approximating the left side of the equation in Theorem \ref{thm:convergence}, is defined as
\begin{equation*}
    \text{RMSE}(\Delta):=\left( \frac{1}{M} \sum_{l=1}^{M} \| X^{(l)}_{t^*} - \tilde{X}^{(l)}_{t^*} \|^2 \right)^{1/2},
\end{equation*}
where $X^{(l)}_{t^*}$ and $\tilde{X}^{(l)}_{t^*}$ denote the values at a fixed time $t^*$ of the $l$-th simulated trajectory of the true process and its numerical approximation, respectively. The integer $M$ is the total number of simulated differences. The value of the true process $X^{(l)}_{t^*}$ is obtained from the EM scheme, using the small time step $\Delta=10^{-4}$. The number of simulations is fixed to $M=10^3$ and $t^*=1$. 

We report the RMSE in Figure \ref{fig:RMSE}, where the $x$-axis corresponds to the logarithm (base 10) of the time step $\Delta$ and the $y$-axis corresponds to the logarithm (base $10$) of the $\text{RMSE}$. The theoretical rate of convergence obtained in Theorem \ref{thm:convergence} (all considered schemes converge with order $1$) is confirmed empirically. 
While the Lie-Trotter splitting and the EM scheme show a similar RMSE for varying $\Delta$, the RMSE obtained for the Strang splitting method is significantly smaller for all $\Delta$ under consideration, implying a higher efficiency of that scheme. We stress, however, that from the fact that the rate of convergence is the same, it does not follow that they share the same qualitative properties when the step size $\Delta$ is fixed.
\begin{figure}
\begin{centering}
    \includegraphics[width=0.6\textwidth]{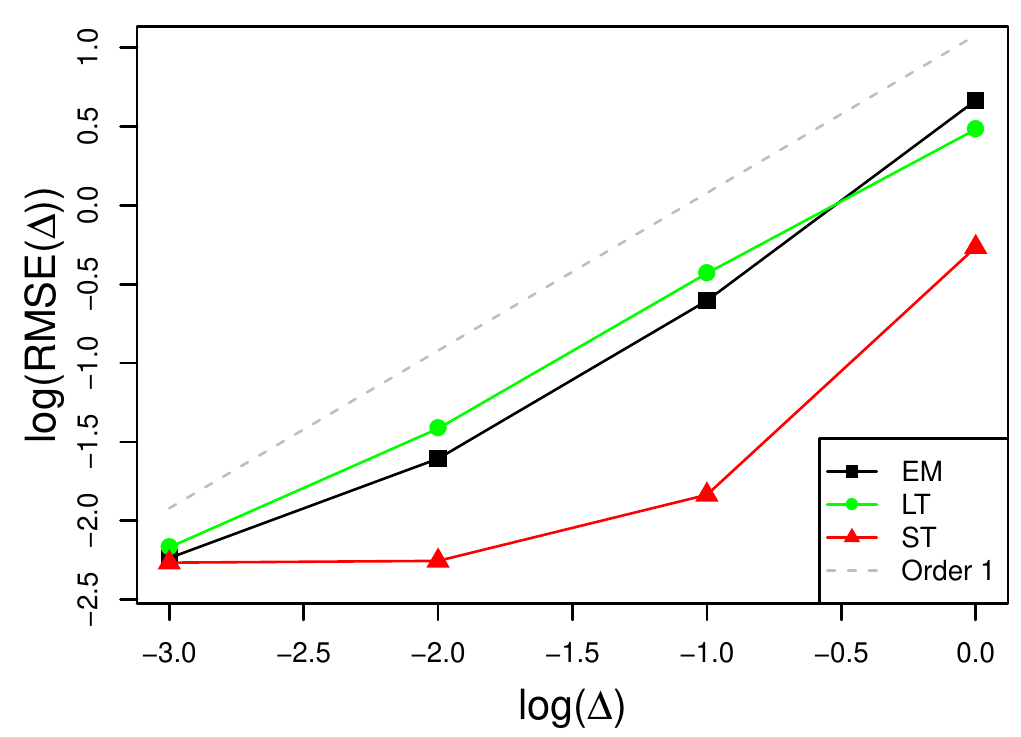}	
	\caption{Mean-square order of convergence. The reference solution is  obtained with the Euler-Maruyama method and the small time step $\Delta=10^{-4}$.  The numerical solutions  are calculated for $\Delta=10^{-3},10^{-2},10^{-1},10^0$. The log is with base $10$, $t^*=1$ and $M=10^3$.}
	\label{fig:RMSE}
\end{centering}
\end{figure}

\subsubsection{Illustration of the qualitative properties of the splitting schemes}
Now we illustrate how the proposed splitting schemes preserve the structure (e.g., the moments and the underlying invariant distribution) of the process, even for large values of $\Delta$, while the EM method may fail in doing so.
We start with studying sample trajectories (see Figure \ref{fig:paths_EM_SP}). All three methods yield a comparable performance when $\Delta=0.01$. For $\Delta = 0.5$, the EM scheme preserves the oscillations, but does not preserve the amplitude. The behaviour of the inhibitory population is less accurately approximated  than the excitatory one. 
This problem aggravates as $\Delta$ increases further to $0.7$. 
An interesting observation is that, for  {time steps $\Delta$ not ``small enough", the Euler-Maruyama scheme may not preserve the mean of the process (see also Figure \ref{fig:densities_EM_SP}). Indeed, it  {has been observed}  {that the} Euler-Maruyama  {method} preserves the first, but not the second moments (see, e.g., \cite{Ableidinger2017,Higham2004}).} In other words, the amplitude of the oscillations grows, but the main is unchanged. In our case, however, since the trajectories are bounded by $0$ from below or above (depending on the sign of $c_k$), {the} increased amplitude introduces  {also} a bias in the first moment.  {Thus}, the Euler-Maruyama approximation of system \eqref{eq:hawkes_approximation} does neither preserve the first nor the second moments.
In contrast, the
Lie-Trotter and Strang splitting schemes show a comparably good performance. 
However, the Lie-Trotter splitting is less accurate in reproducing the delay between the current state of the process (black line) and the memory variables (grey lines) in the beginning of the interval, where the amplitude of the oscillations is large (see also Figure \ref{fig:densities_EM_SP}). 
\begin{figure}
	\centering	
	\subfigure{\includegraphics[width=1.0\textwidth]{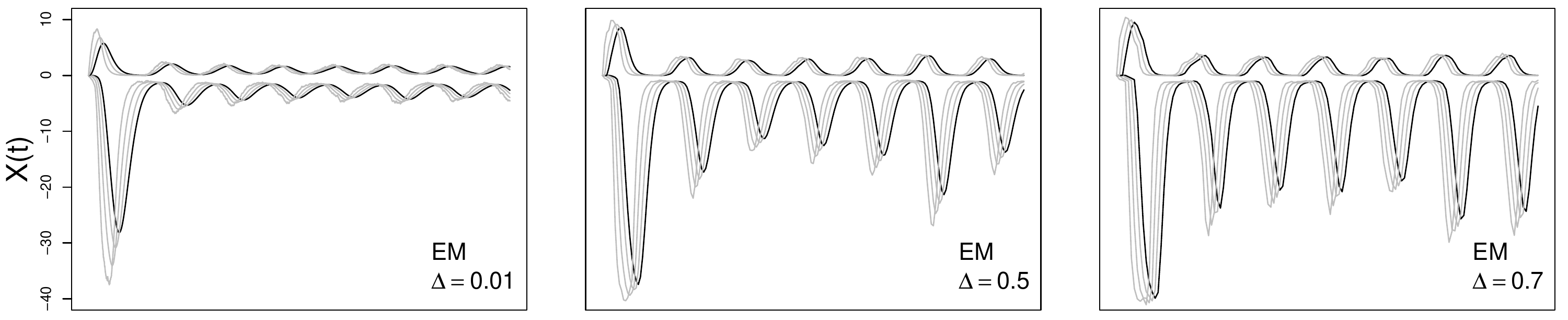}}	
	\subfigure{\includegraphics[width=1.0\textwidth]{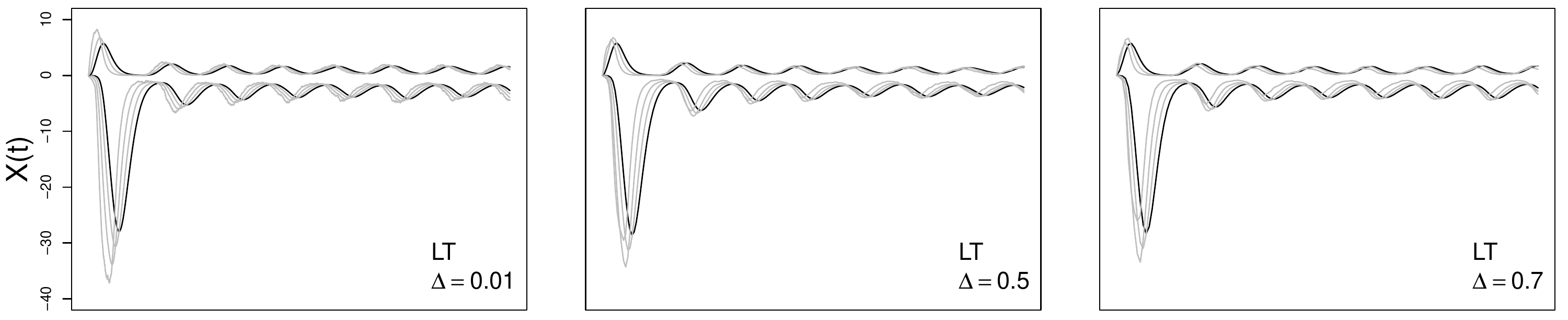}}
	\subfigure{\includegraphics[width=1.0\textwidth]{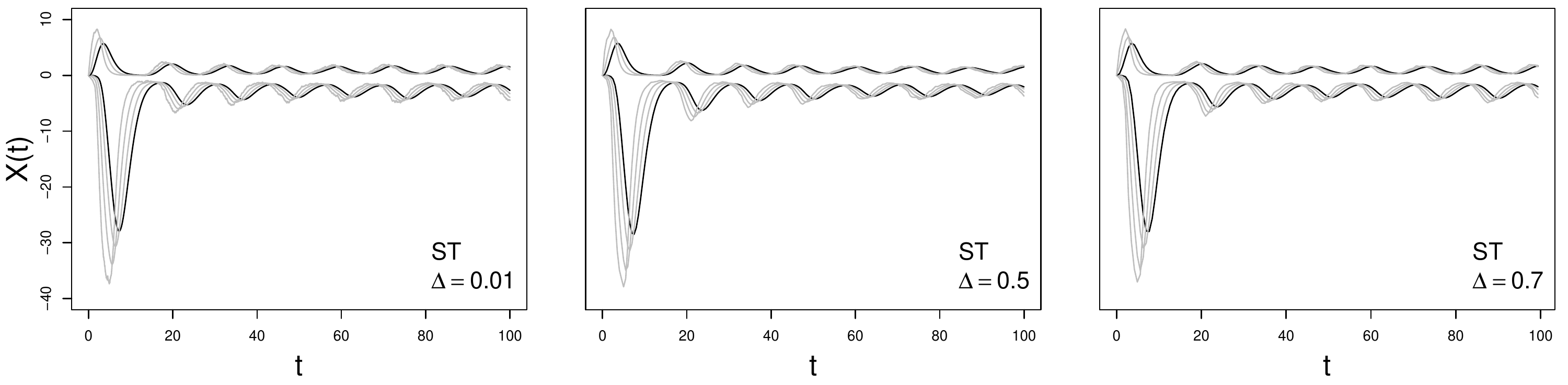}}	
	\caption{Sample trajectories of the system, simulated with the Euler-Maruyama scheme (top), the Lie-Trotter (middle) and the Strang (bottom) splitting scheme for varying $\Delta$.} 
	\label{fig:paths_EM_SP}
\end{figure}

The difference between the schemes becomes clearer as we look at the phase portrait of the system (Figure \ref{fig:phase_EM_SP}). 
We observe again that both splitting schemes yield satisfactory approximations (for all $\Delta$ under consideration), the Strang approach slighly outperforming the Lie-Trotter method. In contrast, the phase portrait obtained with the EM approximation fails to reproduce the behaviour of the process for $\Delta = 0.5$ or $0.7$. 
\begin{figure}
	\centering	
	\subfigure{\includegraphics[width=1.0\textwidth]{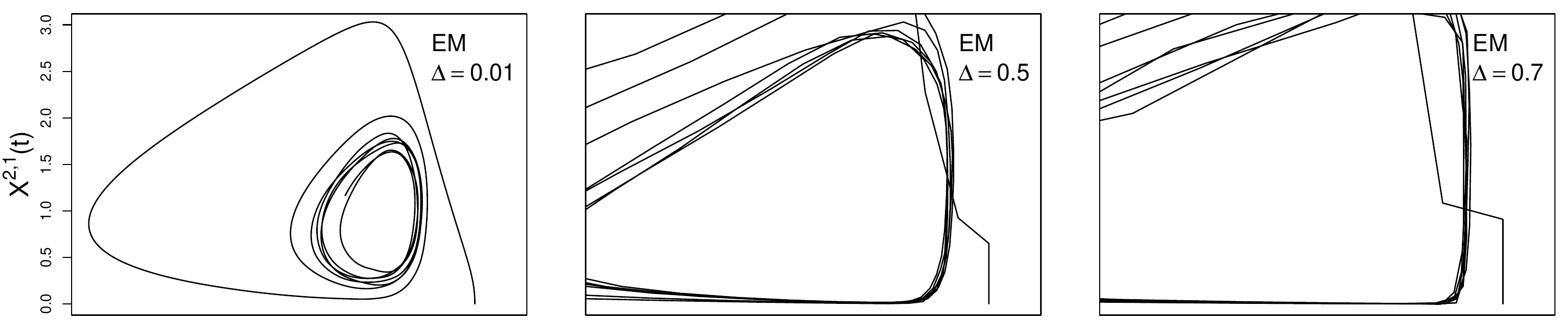}}	
	\subfigure{\includegraphics[width=1.0\textwidth]{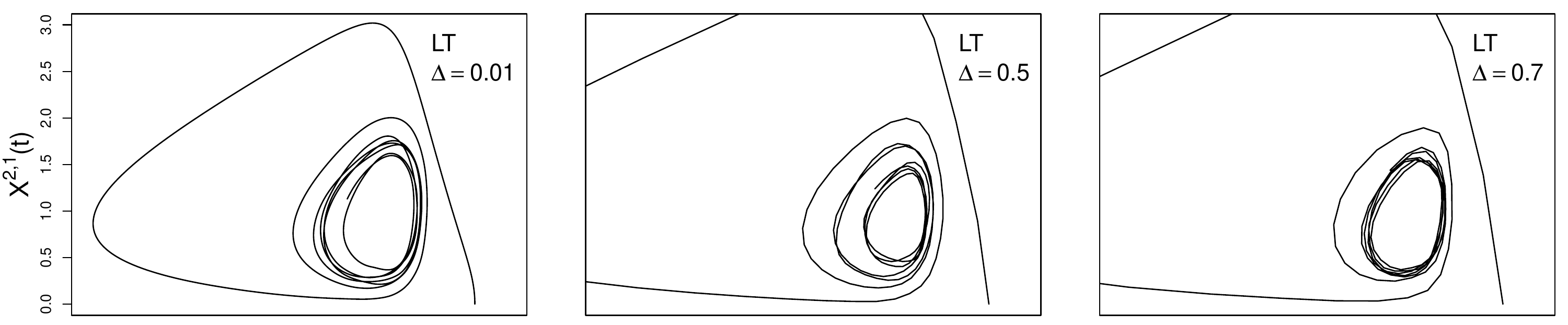}}
	\subfigure{\includegraphics[width=1.0\textwidth]{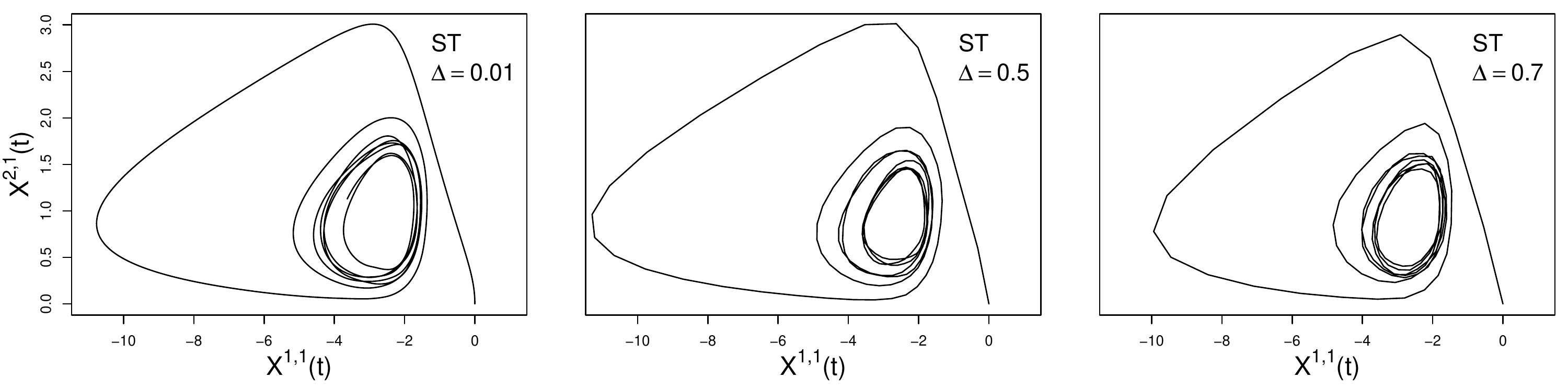}}	
	\caption{Phase portrait of the main variables, simulated with the Euler-Maruyama scheme (top), the Lie-Trotter (middle) and the Strang (bottom) splitting scheme for varying $\Delta$ and $x_0=(0,0,-3.5,-4,0,1.3,1.1)$.} 
	\label{fig:phase_EM_SP}
\end{figure}

Similar conclusions can be drawn from Figure \ref{fig:densities_EM_SP}, where we visualize the marginal densities of the process. Each visualized density is estimated with a standard kernel density estimator, based on a simulated long-time trajectory ($T=10^5)$ of each variable of the process. 
We observe again that the EM method may not preserve the mean of the process (red dashed vertical lines). Moreover, the EM scheme may even suggest a transition from a unimodal to a bimodal density as $\Delta$ increases.

\begin{figure}
	\centering	
	\subfigure{\includegraphics[width=1.0\textwidth]{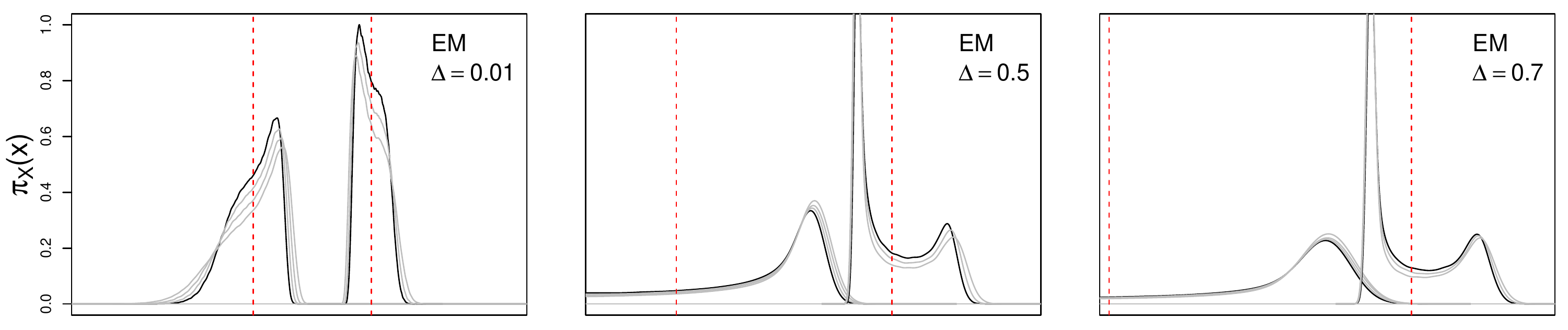}}	
	\subfigure{\includegraphics[width=1.0\textwidth]{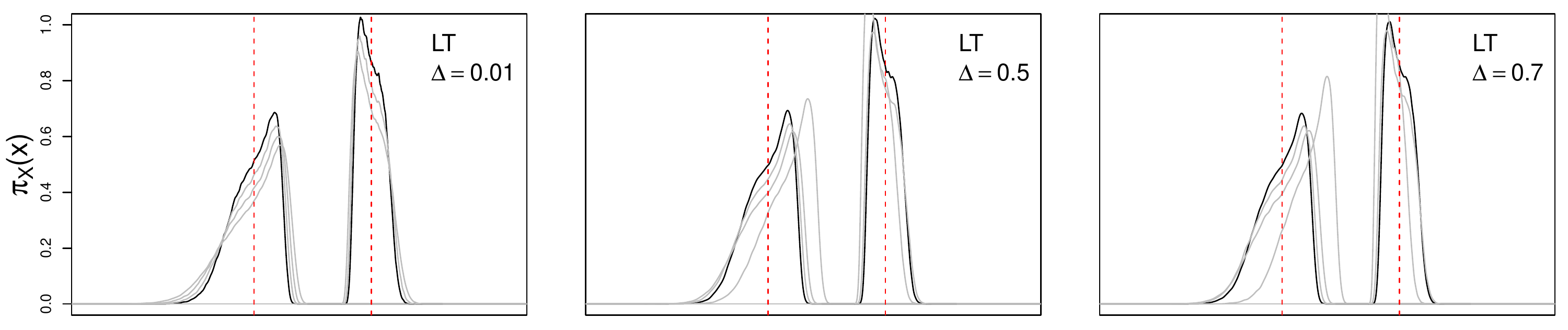}}
	\subfigure{\includegraphics[width=1.0\textwidth]{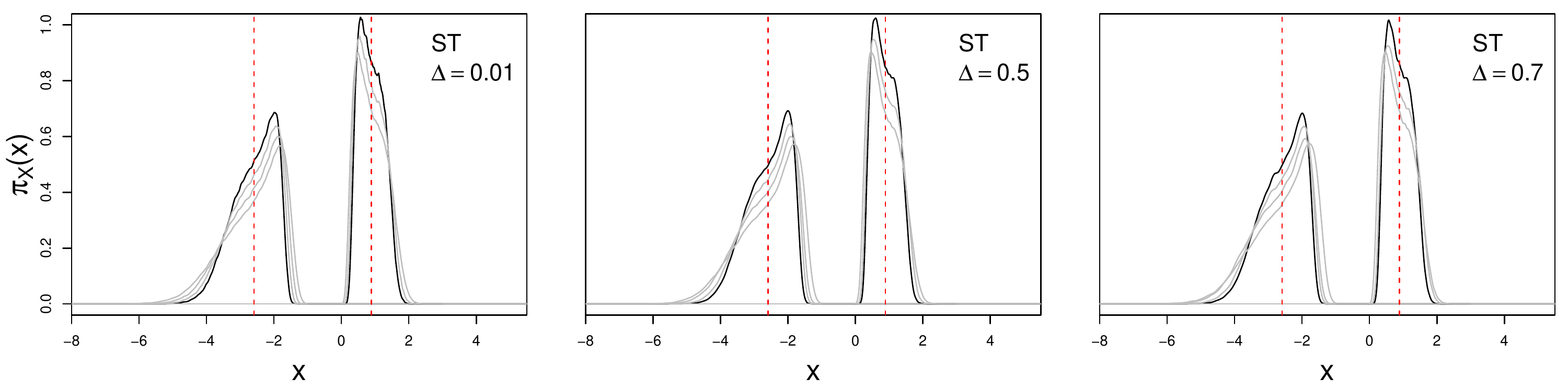}}	
	\caption{Empirical density of the system, simulated with the Euler-Maruyama schemes (top), the Lie-Trotter (middle) and the Strang (bottom) splitting scheme for varying $\Delta$ and $T = 10^5$. The red dashed vertical lines denote the mean of the main variables.}  
	\label{fig:densities_EM_SP}
\end{figure}

\subsection{Comparison of the PDMP and the diffusion}

Now we are interested in comparing the PDMP process $\bar{X}$, simulated with the thinning algorithm detailed in Section \ref{section:simulation_PDMP}, with the diffusion $X$, simulated with the property-preserving Strang splitting scheme introduced in Section \ref{sec:simulation_methods}. We simulate the trajectories of the diffusion process with the Strang splitting scheme, since it has shown the best performance in the previous section. 

\subsubsection{Execution time}
As a first step we are interested in the execution time. We compare the numerical cost of the simulation of the process $\bar{X}$ with two different intensity bounds (based on Lemmas \ref{lem:bound:intensity} and \ref{lem:bound:intensity:local}) to the simulation of the diffusion $X$ with the Strang splitting scheme. 

We set $t_{\max}=100$ and vary the total number of neurons, taking $N = 20, 50, 100, 150, 200$ and $N_1 = N_2$. In the case of the diffusion simulation, the parameter $N$ does not influence the computational cost. Thus, we report the execution time for the diffusion simulation only for $N=200$, taking $\Delta = 0.1$ and report it as a reference value.   
The time step $\tilde\Delta$ for the thinning procedure is defined in an adaptive way within the \textbf{while} loop of Algorithm \ref{algo:simulation:point:process}. In each step we use the last computed value of the intensities $\lambda_k$ and set $\tilde\Delta$ equal to $(N_1\lambda_1+N_2\lambda_2)^{-1}$. This choice takes into account the scaling with respect to the number of neurons and the dynamics of the intensities. For instance, $\bar{X}^{2,1}$ roughly belongs to $[0,2]$ (see Figure \ref{fig:PDMP_and_diffusion}) such that the intensity of population roughly belongs to $[1,7]$ (with the intensity functions defined in \eqref{eq:f1:f2}). 

In Figure \ref{fig:elapsed_time}, two different sets of intensity functions, linear ones and exponential ones, are studied. The mean execution time (over 100 realizations) in seconds required to simulate the process on interval $[0,t_{\max}]$, using the bounds $\tilde{f}(x)$ and $\tilde{f}^\Delta(x)$ are plotted.  
\begin{figure}
\begin{centering}
    \includegraphics[width=0.48\textwidth]{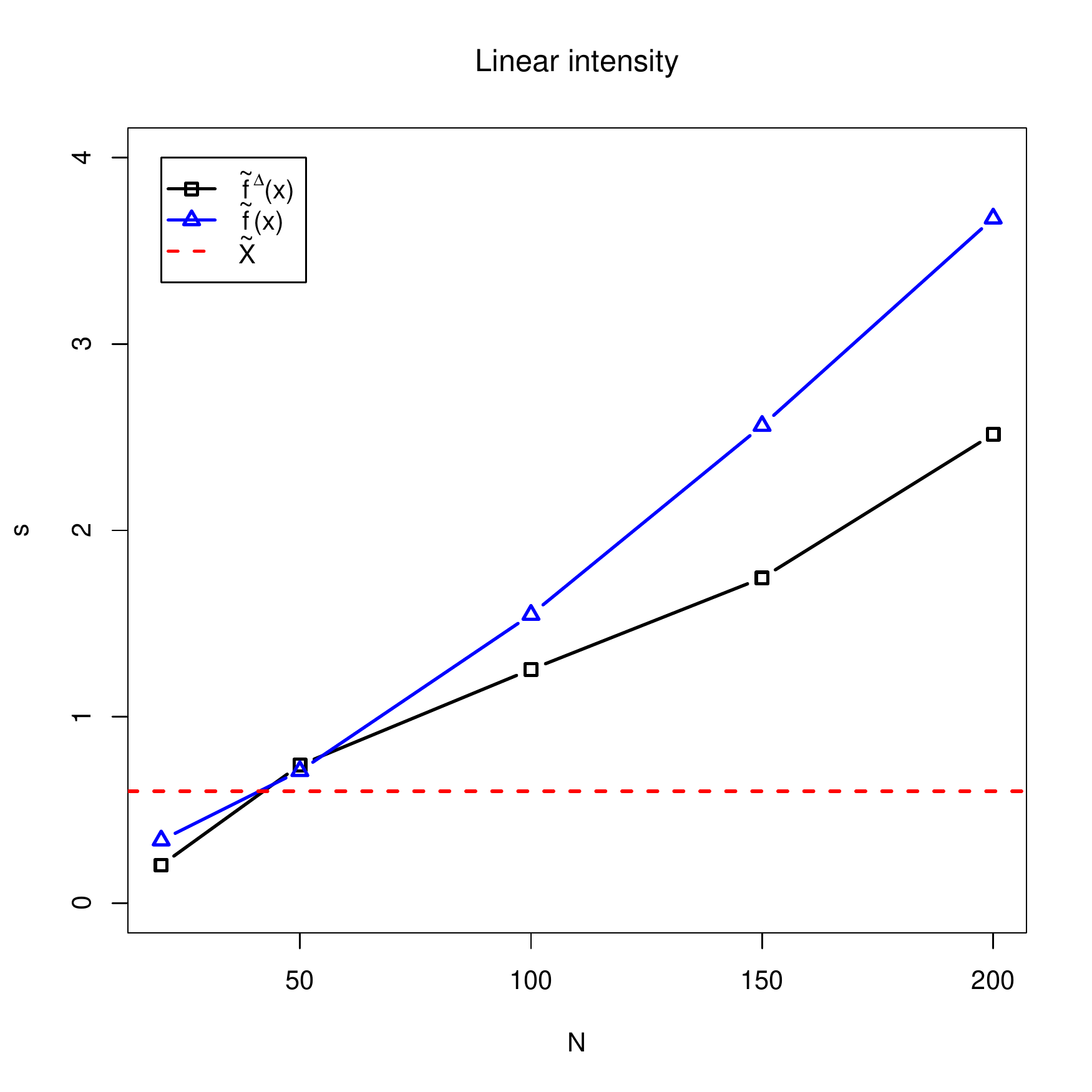}
    \includegraphics[width=0.48\textwidth]{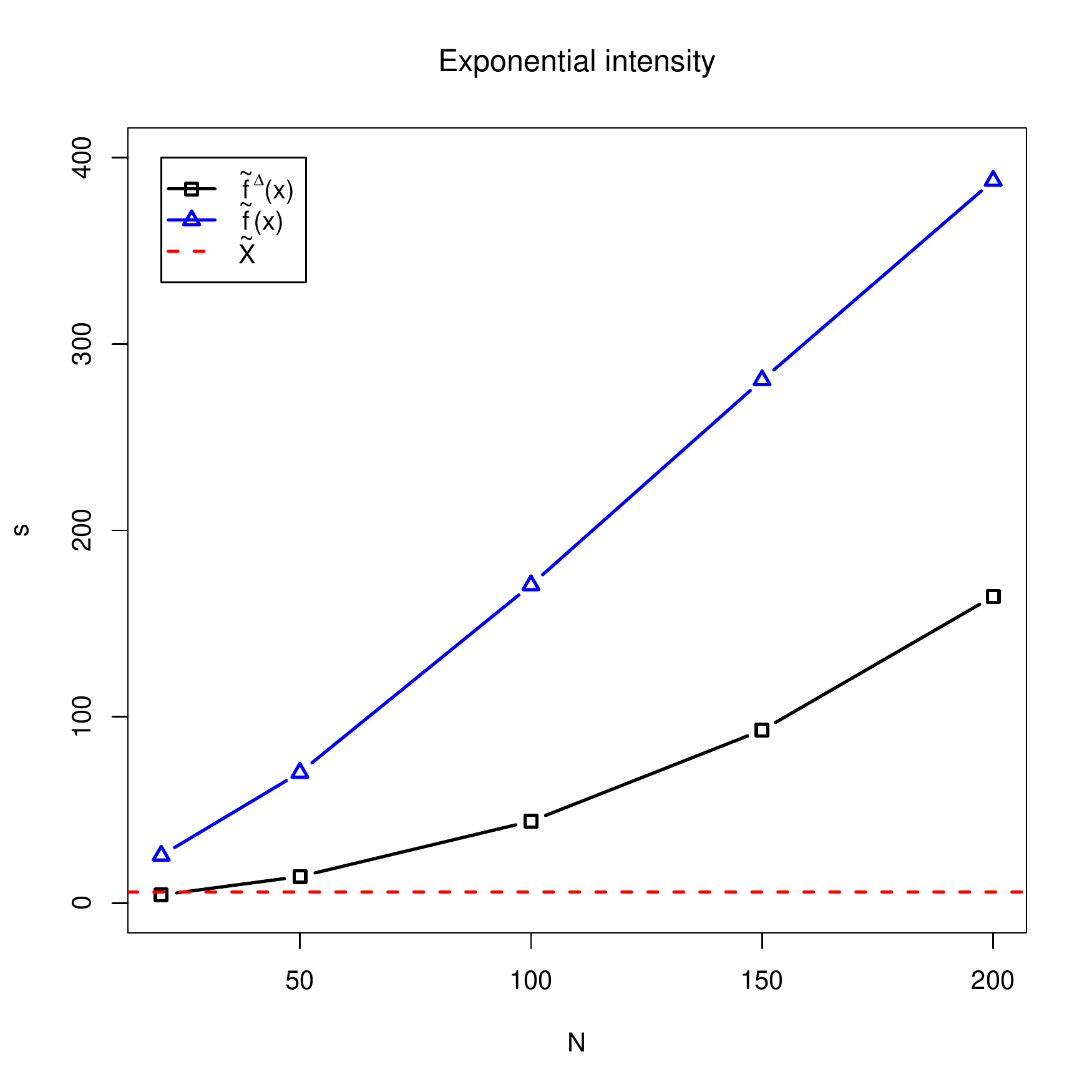}	
	\caption{Mean execution time for the PDMP (solid line) and diffusion (dashed line) simulation for $t_{\max}=100$ over 100 realizations.
	Right panel: $f_1(x) = f_2(x) = \min\{1+x\mathds{1}_{[x>0]}, 10\}$. Left panel: $f_1$ and $f_2$ are given by Equation \eqref{eq:f1:f2}. The rest of the parameters are given in the beginning of Section \ref{sec:simulation_study}.}
	\label{fig:elapsed_time}
\end{centering}
\end{figure}

Note that there is almost no difference in the performance of the algorithm with different bounds in the linear case (left panel of Figure \ref{fig:elapsed_time}).  That means that the bound obtained in Lemma \ref{lem:bound:intensity} is sharp enough. Note also that since $f^{\max}_k = 10$, there occur only a few spikes and the process is simulated very fast. However, in the case of an exponential intensity (right panel of Figure \ref{fig:elapsed_time}), the execution time drastically increases. The process is simulated at least twice faster with the local bound. The main reason is that the local bound $\tilde{f}^\Delta_k(x)$ rejects around 2\% of points, while the $\tilde{f}_k(x)$ rejects around 90\%. In general, we can conclude that the execution time  depends linearly on the number of neurons for both the local and the general bound. Disregarding the bound chosen, both algorithms cannot compete with the time required for simulating the diffusion. For $\Delta = 0.1$ and $T=100$ the average running time with the exponential firing rate function is equal to $0.598$s (with standard deviation $0.12$s). For the linear one it is $0.597$s (with standard deviation $0.15$s). Thus, the execution time for the diffusion approximation does not depend on the firing rates.


Finally, a summary of the performances of both frameworks (diffusion and PDMP), with respect to the parameters of the model, is given below.
\begin{itemize}
    \item In both cases, the execution time increases as the dimension of the system grows, i.e., as $\eta_k$ increases.
    \item For the diffusion, the execution time depends, in a linear manner, on the step size $\Delta$.
    \item For the PDMP, the execution time mainly depends, in a linear manner, on the number $N$ of neurons. To be precise, it also depends on the temporal mean value of the intensities of the two populations, which in turn depends, in a complex non-linear manner, on the parameters $\nu_k$, $\eta_k$ and $f_k$.
    \item Unless $N$ very small, the simulation of the diffusion requires much less computational cost than that of the PDMP. 
\end{itemize}

\subsubsection{Qualitative properties}

It remains to determine if the stochastic diffusion can really catch the behaviour of the underlying PDMP. To get an intuitive idea of how different processes behave when the number of neurons changes we look at some sample trajectories. We take one realisation of the PDMP and the diffusion process on a time interval of length $T = 300$ and plot them on Figure \ref{fig:PDMP_and_diffusion}, cutting the initial part in order to observe the process in its oscillatory regime. For simplicity, we focus only on the second (excitatory) population. The trajectories in the top panel are simulated with $N_2=10$, those in the middle panel with $N_2=50$ and those in the bottom panel with $N_2=100$.

Let us mention that Figure \ref{fig:PDMP_and_diffusion} is not an illustration of Theorem \ref{thm:strong:approx:diffusion}. Indeed, the trajectories are not coupled in such a way that \eqref{eq:result:strong:approx} is satisfied. Up to our knowledge, there is no such result in the literature and the coupling involved in the proof of Theorem \ref{thm:strong:approx:diffusion} is not explicit. However, the figure illustrates the fact that the fluctuations of both trajectories vanish as $N$ goes to infinity and that both converge to the solution of the ODE \eqref{eq:ODE}. 

\begin{figure}
\begin{centering}
    \includegraphics[width=\textwidth]{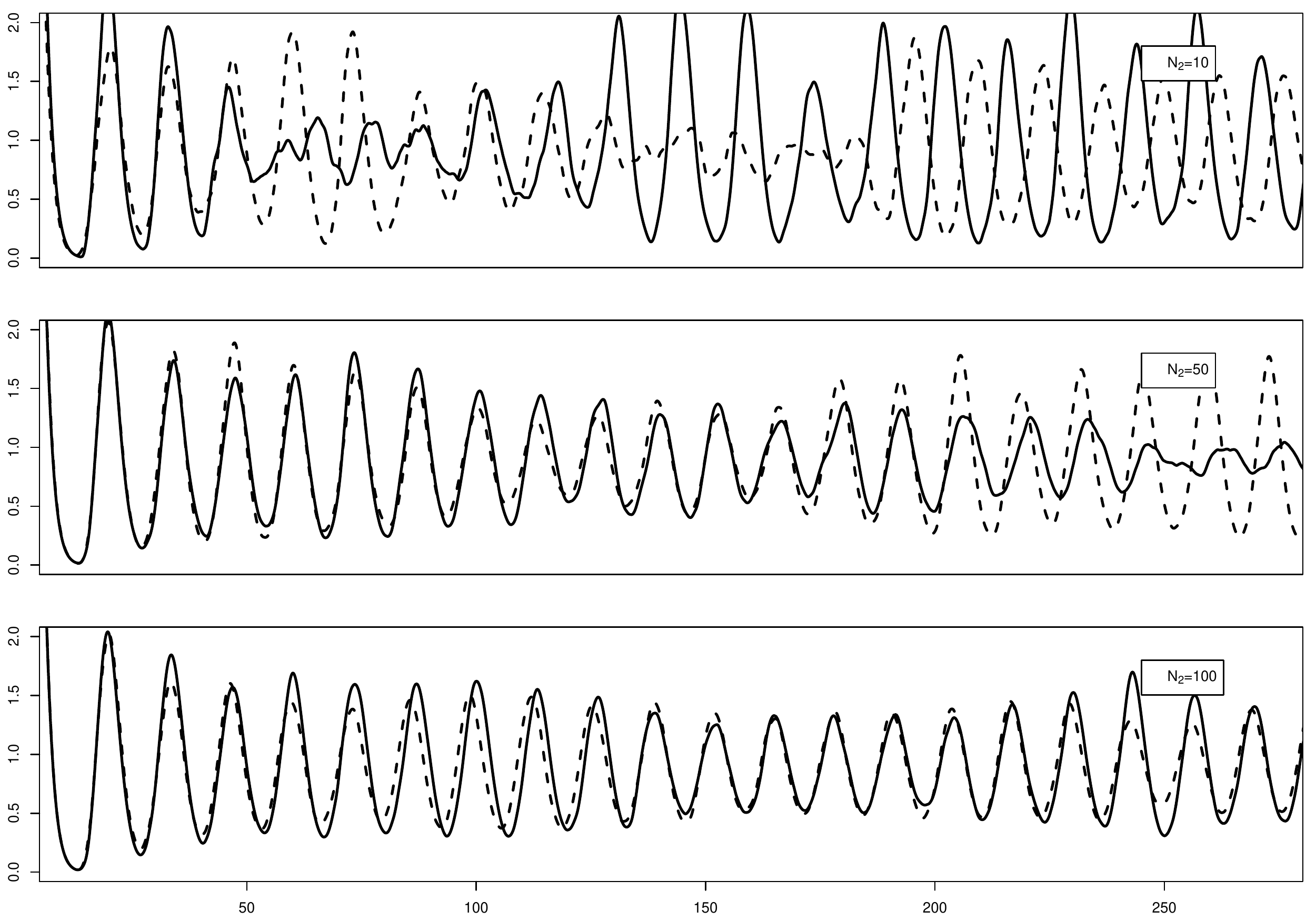}	
	\caption{Sample trajectories of the PDMP and the diffusion for varying $N$ (excitatory population). Solid line: main variable of $\bar{X}$, dashed line: main variable of $X$ (simulated with the Strang splitting scheme, using $\Delta = 0.01$).}
	\label{fig:PDMP_and_diffusion}
\end{centering}
\end{figure}
As a final step, we are interested in the long-time behaviour of the processes. We simulate both processes ($\bar{X}$ and $X$) on a long-time interval, taking $T = 10^5$ and report the respective marginal empirical densities in Figure \ref{fig:PDMP_and_diffusion_density}. The densities of the PDMP are plotted with solid lines and those of the diffusion with dashed lines. Even for small $N$, the difference between the densities is negligible and their means are almost overlapping. As the number of neurons $N$ increases, we observe that the empirical densities converge to some compactly supported distribution. Note that the mean-field limit is given by the ODE \eqref{eq:ODE} as illustrated in Figure \ref{fig:PDMP_and_diffusion}. Thus we expect that the support of the limit distribution is given by the amplitude of the solution of the ODE. 

\begin{figure}
\begin{centering}
    \includegraphics[width=0.48\textwidth]{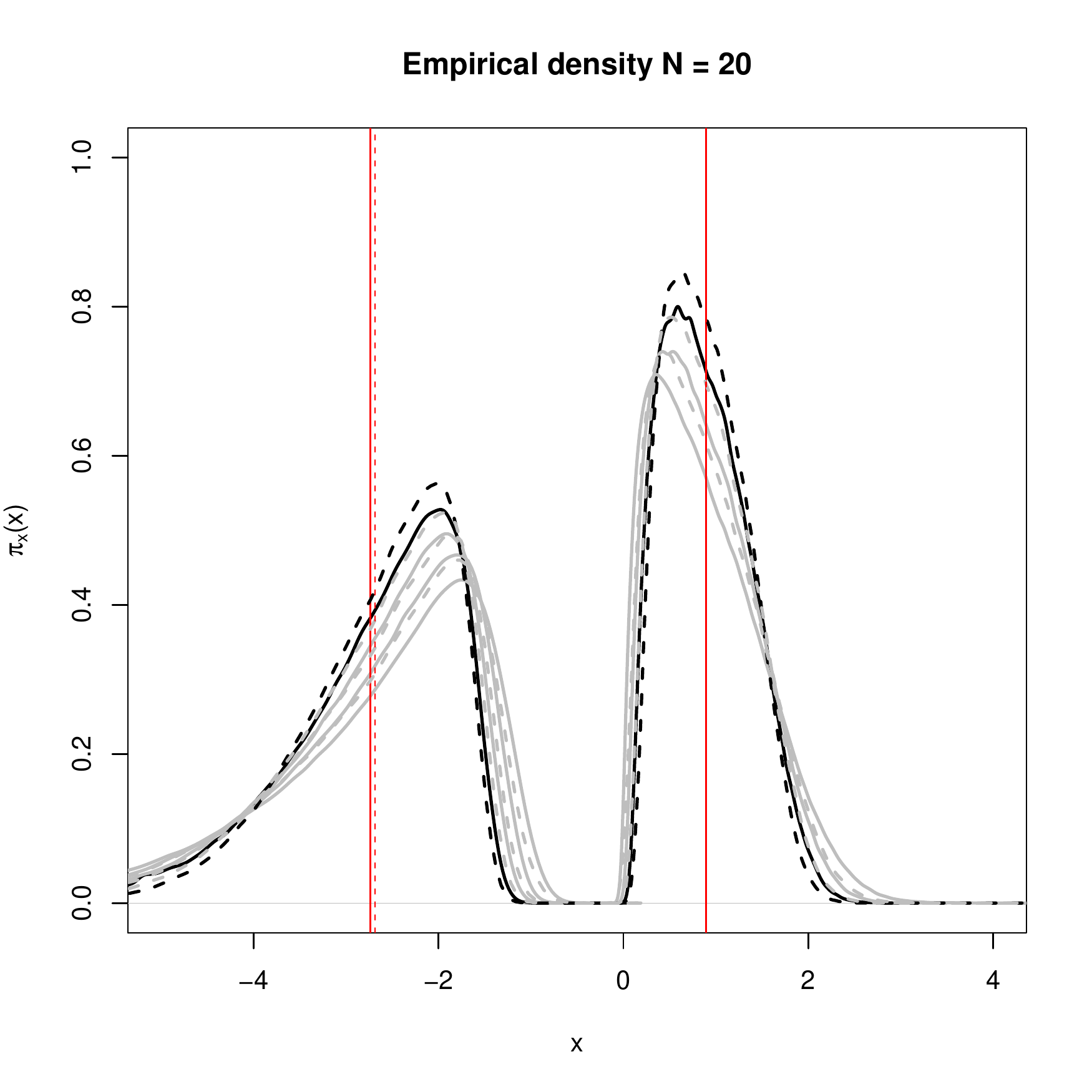}	
    \includegraphics[width=0.48\textwidth]{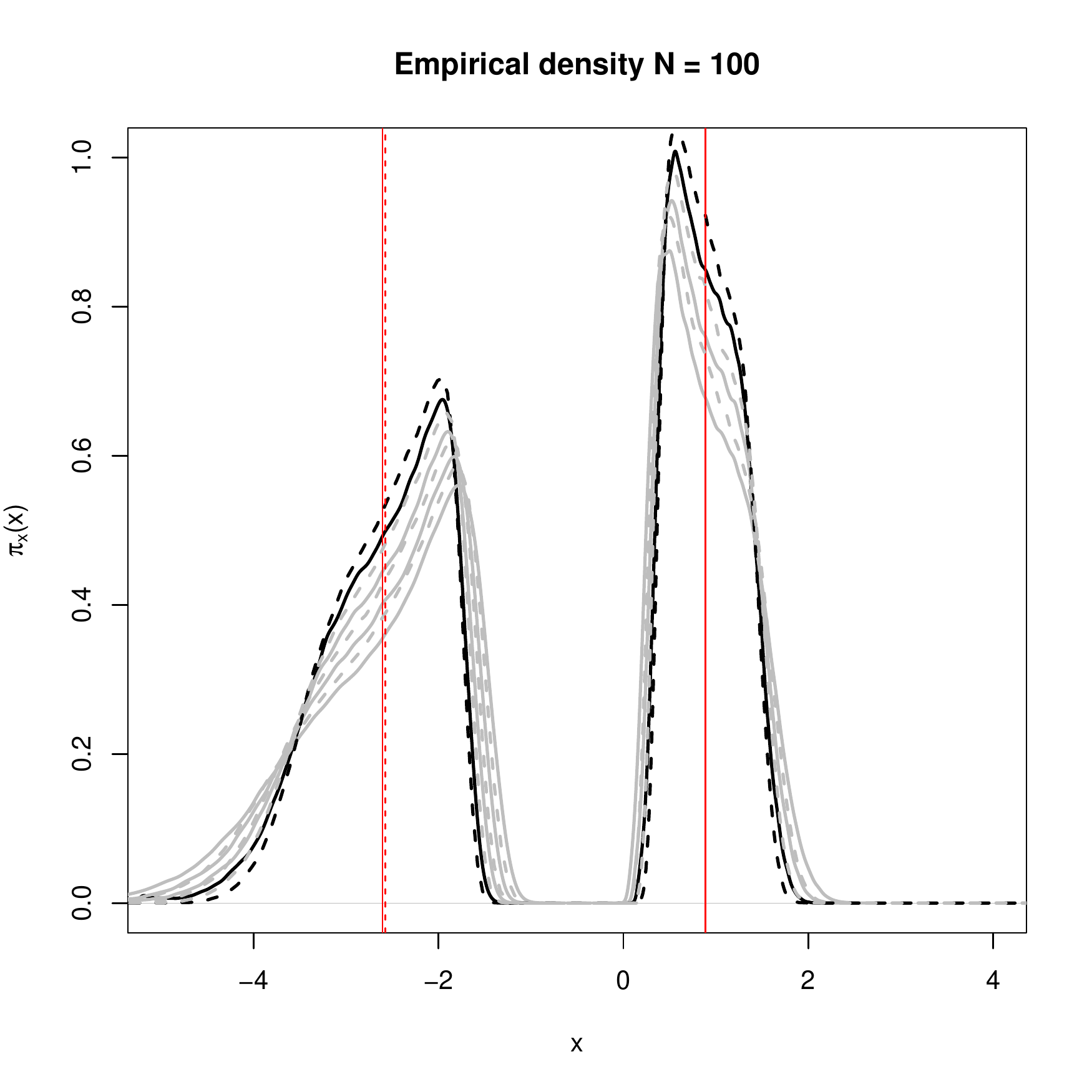}	
	\caption{Empirical density of PDMP and the diffusion for $N=20$ (left) and $N=100$ (right).  Solid line: $\bar{X}$, dashed line: $X$ (simulated with the splitting scheme, $\Delta = 0.1$, $T=10^5$). The red solid and dashed vertical lines denote the mean of the respective main variables.	}
	\label{fig:PDMP_and_diffusion_density}
\end{centering}
\end{figure}

\section*{Conclusions}

This work is thought to complement the papers by \cite{Ditlevsen2017eva} and \cite{Duarte2019}. First, we bridge the gap between the piece-wise deterministic Markov process \eqref{eq:PDMP_flow} and the solution of SDE \eqref{eq:hawkes_approximation} by proving a strong error bound on the distance between the two. Second, moment bounds of the diffusion process are derived.

Further, since SDE \eqref{eq:hawkes_approximation} cannot be solved explicitly, two approximation schemes, based on the Lie-Trotter and the Strang splitting approaches, are proposed. They are proved to converge with mean-square order $1$ and to preserve the properties of the model.
In particular, the advantage of the proposed approximation methods is that they make a full use of the matrix exponential $e^{At}$, which describes the flow of the Markovian cascade \eqref{eq:PDMP_flow}. Thanks to this we are able to propagate the noise through all components of the system, thus preserving its hypoellipticity. Moreover, we show that the splitting schemes accurately reproduce the derived first and second moment bounds and that they preserve the ergodicity of the continuous process, even for large values of the discretization step $\Delta$.

These properties are particularly important when embedding the numerical scheme, for instance, into a statistical inference procedure. For example, maximum likelihood estimation techniques require the existence of a non-degenerate covariance matrix of the discretized process. For simulation-based inference methods (see \cite{Buckwar2019}), the performance of the Euler-Maruyama method may be acceptable for ``small enough" time steps. However, the use of smaller time steps drastically increases the computational cost, making the inference based on the Euler-Maruyama method computationally infeasible.
Moreover, even for arbitrary small time steps there is no guarantee that the Euler-Maruyama scheme preserves the model properties.


In addition,
an exact simulation procedure of the Markovian cascade is proposed. A sharp upper bound, in order to get an efficient procedure, is provided and its performance is compared to the one given in \cite{Duarte2019}. When the number of neurons increases, the computational cost required for the PDMP simulation rises rapidly and cannot compete with the simulation of the diffusion via the splitting scheme. 

The Markovian cascade and the diffusion process show a similar behaviour. In particular, they possess matching empirical densities.
Thus, we conclude that the diffusion process describes the behaviour of the original neuronal model at a very good precision and at negligible computational cost, compared to the PDMP. 

\section*{Acknowledgments}
The authors thank Eva L\"ocherbach and Susanne Ditlevsen for their interest in this manuscript and numerous inspiring discussions. The authors also thank Markus Fischer for fruitful discussions about Brownian modulus of continuity.
J.C. was supported by CNRS under the grant PEPS JCJC \emph{MaNHawkes}.
A.M. was supported by LabEx MME-DII and the Laboratoire Jean Kuntzmann. I.T. was supported by the Austrian Science Fund (FWF): W1214-N15, Project DK 14.
A.M. and I.T. were also supported through the programme ``Research in Pairs" by the Mathematisches Forschungsinstitut Oberwolfach in 2020.

\bibliographystyle{apa}
\bibliography{literature}

\appendix
\section{Proofs}
\subsection{Proof of Theorem \ref{thm:strong:approx:diffusion}}
\label{app:strong_bound}

The proof of Theorem \ref{thm:strong:approx:diffusion} is mainly based on two lemmas which are stated before the proof. The first lemma concerns the coupling between a Poisson process and a Brownian motion. Its proof can be found in \citet[Section 5.5]{Ethier2009} (the exponential moments can be deduced from the proof of Corollary 5.5.5).
\begin{lemma}\label{lemma:KMT}
A standard Poisson process $(\Pi_t)_{t\geq 0}$ and a standard one-dimensional Brownian motion $ (B_t)_{t\geq 0}$  can be constructed on the same probability space such that 
$$ \sup_{t \geq 0} \frac{ | \Pi_t -t -  B_t|}{ \log (2 \vee t) } \le \Xi < \infty $$
almost surely, where $\Xi$ is a random variable having exponential moments.
\end{lemma}

The second lemma concerns the modulus of continuity for the Brownian motion. It is stated in \cite{Kurtz1978} without a proof. Hence, for the sake of completeness, we provide a proof which is an adaptation of the arguments presented in the appendix of \cite{fischer2009moments}.
\begin{lemma}\label{lem:mod:cont:BM}
Let $B$ be a standard Brownian motion and $T$ a positive time. Then,
\begin{equation}
    M := \sup_{t,s\leq T} \frac{|B_t - B_s|}{\sqrt{|t-s|(1+\log(T/|t-s|))}},
\end{equation}
is a finite random variable such that $M^2$ has finite exponential moments.
\end{lemma}
\begin{proof}
Thanks to the scaling properties of Brownian motion, it is sufficient to prove the statement for $T=1$.
According to \cite{fischer2009moments}, let $c>1$ and define two increasing functions $\Psi$ and $\mu$ by
\begin{equation}
    \Psi(x) = e^{x^2/2}-1 \quad \text{and} \quad \mu(x) = \sqrt{cx},
\end{equation}
for all $x\geq 0$.
Let now $\xi$ be the random variable defined by
\begin{equation}
    \xi = \int_0^1 \int_0^1 \Psi\left( \frac{|B_t - B_s|}{\mu(|t-s|)}\right) dtds.
\end{equation}
The Garsia–Rodemich–Rumsey inequality \cite[Theorem 2.1.3.]{stroock2007multidimensional} implies that 
\begin{equation}
    |B_t - B_s| \leq 8 \int_0^{|t-s|} \Psi^{-1}\left( \frac{4\xi}{x^2} \right) \mu'(x) dx,
\end{equation}
with $\Psi^{-1}(y) = \sqrt{2\log(1+y)}$ and $\mu'(x) = (\sqrt{c}/2) x^{-1/2}$.
Yet, for $0<x<1$,
\begin{equation}
    \Psi^{-1}\left( \frac{4\xi}{x^2} \right) = \sqrt{2} \sqrt{\log(4\xi+x^2) + 2\log(1/x)}  \leq \sqrt{2} \sqrt{\log(4\xi+1)} + 2 \sqrt{\log(1/x)}.
\end{equation}
Combining the last two equations, one gets for all $h$,
\begin{equation}\label{eq:control:modulus:1}
    \sup_{|t-s|\leq h} |B_t - B_s| \leq 4\sqrt{2c} \sqrt{\log(4\xi+1)} \int_0^{h} \frac{dx}{\sqrt{x}} + 8\sqrt{c} \int_0^h \sqrt{\log(1/x)} \frac{dx}{\sqrt{x}}.
\end{equation}
The second term can be bounded thanks to
\begin{align}
    \int_0^h \sqrt{\log(1/x)} \frac{dx}{\sqrt{x}} &= \int_0^h \left( \sqrt{\log(1/x)} - \frac{1}{\sqrt{\log(1/x)}}\right) + \frac{1}{\sqrt{\log(1/x)}} \frac{dx}{\sqrt{x}}\\
    & \leq 2 \sqrt{h\log(1/h)} + 4 \sqrt{h},
\end{align}
using (when $h>e^{-1}$) the fact that
\begin{equation}
    \int_{e^{-1}}^h \frac{1}{\sqrt{x\log(1/x)}} dx =  \int_{e^{-1}}^h 2\sqrt{x} \frac{1}{x \sqrt{\log(1/x)}} dx \leq 2\sqrt{h} (1 - \sqrt{\log(1/h)}).
\end{equation}
Going back to Equation \eqref{eq:control:modulus:1}, for some constant $C$ which does not depend on $c$, one has that
\begin{equation}
    \sup_{|t-s|\leq h} |B_t - B_s| \leq C \sqrt{c} \left( \sqrt{\log(4\xi+1)} + 1 \right) \sqrt{h (1+\log(1/h))}.
\end{equation}
Note that the random variable $M$ defined in the statement of the lemma satisfies
\begin{equation}
    M = \sup_{0<h<1} \frac{\sup_{|t-s|\leq h} |B_t - B_s|}{\sqrt{h (1+\log(1/h))}} \leq C \sqrt{c} \left( \sqrt{\log(4\xi+1)} + 1 \right),
\end{equation}
so that 
$$\mathbb{E}\left[ e^{\alpha M^2} \right] \leq  \mathbb{E}\left[ e^{2 \alpha c C^2 \left( \log(4\xi + 1) + 1\right)} \right] \leq e^{2 \alpha c C^2} \mathbb{E}\left[ (4\xi + 1)^{2\alpha c C^2}\right].$$
To conclude, we refer to the control of the moments of $\xi$ given in the appendix of \cite{fischer2009moments}. It states in particular that $\mathbb{E}[ (4\xi + 1)^{2\alpha c C^2}]$ is finite as soon as $2\alpha c C^2 < c$ which is granted if we take $\alpha$ small enough.
\end{proof}

Before going through the proof of the Theorem, let us give some alternative representation of Equation \eqref{eq:PDMP_flow} and some sketch of the proof.
Thanks to the time change property of point processes (see \citet[Section II.6.]{bremaud1981point} for instance), there exists two independent standard (i.e., with rate equal to one) Poisson processes $\Pi^1$ and $\Pi^2$ such that $\bar{Z}^k_t = N_k^{-1}\, \Pi^k_{\bar{\Lambda}^k_t}$ where $\bar{\Lambda}^k_t$ is the integrated intensity of $\bar{Z}^k_t$, that is
\begin{equation}
    \bar{\Lambda}^k_t = N_k \int_0^t f_k(\bar{X}^{k+1,1}_s) ds.
\end{equation}
This time-change property is an analogous martingale property to the time-change property for diffusions.
Then, the integrated form of \eqref{eq:PDMP_flow} is given by
\begin{equation}\label{eq:Xbar:evolution:time:change}
    \bar{X}_t = x_0 + \int_0^t A \bar{X}_s ds + c \,\bar{Z}_t = \int_0^t A \bar{X}_s ds + c 
    \begin{pmatrix}
    N_1^{-1}\,\Pi^1_{\bar{\Lambda}^1_t}\\
    N_2^{-1}\,\Pi^2_{\bar{\Lambda}^2_t}
    \end{pmatrix}.
\end{equation}
In a similar way, the SDE can be written with respect to two time-changed Brownian motions and the general idea of the proof is then to couple the standard Poisson processes $\Pi^k$ with the Brownian motions.

\begin{proof}[Proof of Theorem \ref{thm:strong:approx:diffusion}]
It is more convenient to first prescribe the Brownian motions and then couple them with Poisson processes. That is exactly how we proceed below.  Let $Y$ be the solution of \eqref{eq:hawkes_approximation} with respect to some two dimensional Brownian motion $W=(W^1,W^2)^T$. Thanks to the time change property of the Brownian motion (see \citet[Theorem 2.12.]{Ethier2009} for instance), let $B^k$ be the Brownian motion defined by
\begin{equation}
    B^k_t = \int_0^{\tau^k(t)} \sqrt{N_k f_{k}(Y^{k,1}_s)}  dW^{k}_s,
\end{equation}
where $\tau^k(t)$ is the stopping time satisfying
\begin{equation*}
    t = N_k \int_0^{\tau^k(t)} f_{k}(Y^{k,1}_s) ds.
\end{equation*}
Then, $Y$ can be written as follows
\begin{equation}\label{eq:Y:evolution:time:change}
Y_t = x_0+ \int_0^t A Y_s ds + \int_0^t \Gamma 
\begin{pmatrix}
f_2(Y^{2,1}_s)\\
f_1(Y^{1,1}_s)
\end{pmatrix}
ds + \Gamma 
\begin{pmatrix}
N_1^{-1} B^1_{\Lambda^1_t}\\
N_2^{-1} B^2_{\Lambda^2_t}
\end{pmatrix},
\end{equation}
where 
\begin{equation}\label{eq:Y:time:change}
\Lambda^k_t = N_k \int_0^t f_{k}(Y^{k,1}_s) ds.
\end{equation}

We are now in the position to use the coupling with Poisson processes. Let $\Pi^k$ be the Poisson process given by Lemma \ref{lemma:KMT} with associated random variable $\Xi_k$. Now, let $\bar{X}$ be defined as in \eqref{eq:Xbar:evolution:time:change}. 
Then,
\begin{equation}\label{eq:Xbar:evolution:time:change:plus:rest}
\bar{X}_t = x_0 + \int_0^t A \bar{X}_s ds + \int_0^t \Gamma 
\begin{pmatrix}
f_2(\bar{X}^{2,1}_s)\\
f_1(\bar{X}^{1,1}_s)
\end{pmatrix}
ds + \Gamma 
\begin{pmatrix}
N_1^{-1} B^1_{\bar{\Lambda}^1_t} + R^1_t\\
N_2^{-1} B^2_{\bar{\Lambda}^2_t} + R^2_t
\end{pmatrix},
\end{equation}
where 
$$ R^k_t = \frac{1}{N_k} \left( \Pi^k_{\bar \Lambda_t^k } - \bar \Lambda_t^k - B^k_{\bar \Lambda_t^k } \right).$$
Thanks to Lemma \ref{lemma:KMT}, 
\begin{align}
| R^k_t | \leq \frac{1}{N_k} \Xi_k \log (2\vee \bar\Lambda_t^k) &\leq  \Xi_k \left( \frac{\log N_k}{N_k} + \frac{\log t }{N_k} + \frac{1}{N_k} \right),  \nonumber  \\
& \leq C\, \Xi_k \left( \frac{\log N}{N} + \frac{\log t }{N} + \frac{1}{N} \right), \label{eq:bound:rest:term}
\end{align}
for some constant $C$, where we used that $ \bar\Lambda_t^k \le N_k t f_k^{max}$ and $N/N_k$ is bounded for $N$ and $N_k$ large enough. 

Let us denote $G^N(t) = \sup_{s\leq t} N \| \bar{X}_s - Y_s \|$ where $\|\cdot \|$ denotes the sup norm on $\mathbb{R}^\kappa$ here and below. Combining \eqref{eq:Y:evolution:time:change} and \eqref{eq:Xbar:evolution:time:change:plus:rest} as well as using the Lipschitz continuity of $f_k$ (with respect to constant $L_k$) give
\begin{multline}\label{eq:decomposition:pathwise:estimate}
    \| \bar{X}_t - Y_t \| \leq  \int_0^t \| A (\bar{X}_s - Y_s ) \| ds + \max\{|c_1|L_2,|c_2|L_1\} \int_0^t \| \bar{X}_s - Y_s \| ds \\
    + \max_k\left\{ |c_{k-1}| \left( N_k^{-1} \left| B^k_{\bar{\Lambda}_t^k} - B^k_{\Lambda_t^k}\right| + |R^k_t| \right) \right\}. 
\end{multline}
Then, since the operator norm $||A||$ corresponding to the sup norm is finite, Grönwall's lemma yields
\begin{equation}\label{eq:gronwall:gT}
    G^N(T) \leq C_1 \max_k \left\{ \sup_{t\leq T} \left| B^k_{\bar{\Lambda}_t^k} - B^k_{\Lambda_t^k}\right| + N| R^k_t| \right\} e^{C_2 T}
\end{equation}
for two deterministic constants $C_1$ and $C_2$ which do neither depend on $N$ nor on $T$. Hence, it only remains to estimate the Brownian increments. This can be done via the modulus of continuity of Brownian motion. Indeed, for $t\leq T$, $\bar{\Lambda}_t^k$ and $\Lambda_t^k$ are bounded by $N T f_k^{max}$ so Lemma \ref{lem:mod:cont:BM} gives
\begin{equation}
    \left| B^k_{\bar{\Lambda}_t^k} - B^k_{\Lambda_t^k}\right| \leq M_k \sqrt{\left| \bar{\Lambda}_t^k - \Lambda_t^k\right| (1+\log(N f_k^{max} T /|\bar{\Lambda}_t^k - \Lambda_t^k|))},
\end{equation}
where $M_k$ is some random variable defined in the lemma.
For all $a>0$, the function $x\mapsto \sqrt{x(1+\log(a/x))}$ is increasing for $0<x\leq a$ and Lipschitz continuity of $f_k$ gives
$$
|\bar{\Lambda}_t^k - \Lambda_t^k| = N \left| \int_0^t f_k(\bar{X}^{k,1}_s) - f_k(Y^{k,1}_s) ds \right| \leq C \int_0^t G^N(s) ds \leq CTG^N(T)
$$
so that 
\begin{equation}
    \left| B^k_{\bar{\Lambda}_t^k} - B^k_{\Lambda_t^k}\right| \leq M_k \sqrt{CTG^N(T) (1+\log(Nf_k^{max}/CG^N(T)))}.
\end{equation}
On the event where $G^N(T)<1$, \eqref{eq:result:strong:approx} holds. If $G^N(T)\geq 1$ then the equation above implies
\begin{equation}
    \left| B^k_{\bar{\Lambda}_t^k} - B^k_{\Lambda_t^k}\right| \leq M_k \sqrt{CTG^N(T) (1+\log(Nf_k^{max}/C))}
\end{equation}
and so coming back to \eqref{eq:gronwall:gT} one has
\begin{multline}
    G^N(T) \leq C_1 \left( M \sqrt{CT (1+\log(N f^{max} /C))}\sqrt{G^N(T)} \right. \\
    + \left. N \max_k \sup_{t\leq T} | R^k_t| \right) e^{C_2 T},
\end{multline}
with $f^{max}=\max\{f_1^{max},f_2^{max}\}$ and $M=\max\{M_1,M_2\}$.
The inequality above is of order 2 with respect to $x=\sqrt{G^N(T)}$. Yet, the positive values of $x$ such that $p(x) = x^2 +bx +c $ is negative are such that $x^2 \leq b^2 - c$. Hence,
\begin{equation}
    G^N(T) \leq C \left( M^2 T (1+\log(N f^{max}/C)) + N\max_k \sup_{t\leq T} | R^k_t| \right) e^{2C_2 T}.
\end{equation}
Finally, \eqref{eq:result:strong:approx} follows from the control of $|R^k_t|$ given by \eqref{eq:bound:rest:term}.
\end{proof}

\subsection{Proof of Theorem \ref{thm:second_moment_bounds_process}}
\label{app:second_moment}

 Recall that the components of the process $X^k$ are given by
\begin{align*}
X^{k,j}_t &= \left(e^{At} X_{0}\right)^{k,j}+\int_0^t c_k f_{k+1}(X^{{k+1},1}_s)  \frac{e^{-{\nu_k}(t-s)}}{(\eta_k+1-j)!} (t-s)^{\eta_k+1-j} ds \\ &+\frac{1}{\sqrt{N}} \int_0^t \frac{c_k}{\sqrt{p_{k+1}}}\sqrt{f_{k+1}(X^{{k+1},1}_s)}  \frac{e^{-{\nu_k}(t-s)}}{(\eta_k+1-j)!} (t-s)^{\eta_k+1-j}  dW^{k+1}_s.
\end{align*}
Squaring the above expression yields
\begin{equation*}
    (X^{k,j}_t)^2  = T_1(t) + T_2(t)+T_3(t)+ T_4(t) + T_5(t)+T_6(t),
\end{equation*}
where
\begin{align*}
    T_1 & = \left((e^{At}x_0)^{k,j}\right)^2,\\ 
    T_2 & = \left(\int_0^t c_k f_{k+1}(X^{k+1,1}_s)  \frac{e^{-\nu_{k}(t-s)}}{(\eta_k+1-j)!} (t-s)^{\eta_k+1-j} ds \right)^2,\\ 
    T_3 & = \left(\frac{1}{\sqrt{N}} \int_0^t \frac{c_k}{\sqrt{p_{k+1}}}\sqrt{f_{k+1}(X^{k+1,1}_s)} \frac{e^{-\nu_{k}(t-s)}}{(\eta_k+1-j)!} (t-s)^{\eta_k+1-j}  dW^{k+1}_s \right)^2,\\ 
    T_4 & = 2 (e^{At}x_0)^{k,j} \int_0^t c_k f_{k+1}(X^{k+1,1}_s) \frac{e^{-\nu_{k}(t-s)}}{(\eta_k+1-j)!} (t-s)^{\eta_k+1-j} ds,\\
    T_5 & = 2 \int_0^t c_k f_{k+1}(X^{{k+1},1}_s) \frac{e^{-\nu_{k}(t-s)}}{(\eta_k+1-j)!} (t-s)^{\eta_k+1-j} ds \\ & \quad \cdot \frac{1}{\sqrt{N}} \int_0^t \frac{c_k}{\sqrt{p_{k+1}}}\sqrt{f_{k+1}(X^{{k+1},1}_s)} \frac{e^{-\nu_{k}(t-s)}}{(\eta_k+1-j)!} (t-s)^{\eta_k+1-j}  dW^{k+1}_s,\\ 
    T_6 & = 2 (e^{At}x_0)^{k,j} \frac{1}{\sqrt{N}} \int_0^t \frac{c_k}{\sqrt{p_{k+1}}}\sqrt{f_{k+1}(X^{{k+1},1}_s)} \frac{e^{-\nu_{k}(t-s)}}{(\eta_k+1-j)!} (t-s)^{\eta_k+1-j}  dW^{k+1}_s.\\ 
\end{align*}
First, we note that $\mathbb{E}[T_6(t)]=0$ and that $\mathbb{E}[T_1(t)]=T_1(t)$. Since the intensity function is bounded by $0<f_{k+1}\leq f_{k+1}^{\max}$, we have that
\begin{equation*}
\mathbb{E}[T_4(t)]\leq \max\left\{0, \frac{c_k f_{k+1}^{\max}}{(\eta_k+1-j)!} \right\}2(e^{At}x_0)^{k,j} \int_0^t   {e^{-\nu_{k}(t-s)}} (t-s)^{\eta_k+1-j} ds.
\end{equation*}
Further, applying It\^{o}'s isometry gives  
\begin{equation*}
    \mathbb{E}[T_3(t)]=\frac{c_k^2}{N p_{k+1} ((\eta_k+1-j)!)^2} \int_0^t \mathbb{E}[f_{k+1}(X^{{k+1},1}_s)]  {e^{-2{\nu_k}(t-s)}} (t-s)^{2(\eta_k+1-j)}  ds.
\end{equation*}
Using again the fact that $f_{k+1}<f_{k+1}^{\max}$ results in
\begin{equation*}
    \mathbb{E}[T_3(t)] \leq \frac{1}{N} \frac{c_k^2}{p_{k+1}}  \frac{f_{k+1}^{\max}}{((\eta_k+1-j)!)^2} \int_0^t  e^{-2{\nu_k}(t-s)} (t-s)^{2(\eta_k+1-j)}  ds.
\end{equation*}
Moreover, since $f_{k+1}$ is bounded, also $(T_2(t))^2$ is bounded, and thus it follows from the proof of Theorem \ref{thm:moment_bounds_process} that
\begin{equation*}
    \mathbb{E}[T_2(t)]\leq  \left(\frac{c_k f_{k+1}^{\max}}{(\eta_k+1-j)!} \right)^2 \left( \int_0^t  e^{-{\nu_k}(t-s)} (t-s)^{(\eta_k+1-j)} ds \right)^2.
\end{equation*}
Applying the Cauchy-Schwarz inequality gives that
\begin{equation*}
    \mathbb{E}[T_6(t)] \leq 2 \left( \mathbb{E}[T_2(t)] \mathbb{E}[T_3(t)] \right)^{1/2}.
\end{equation*}
Combining the above results and using that 
\begin{equation*}
    \int_0^t  e^{-2{\nu_k}(t-s)}  (t-s)^{2(\eta_k+1-j)} ds = \frac{(2(\eta_k+1-j))!}{(2\nu_k)^{2(\eta_k+1-j)+1}} \left[ 1-e^{-2\nu_k t} \sum_{l=0}^{2(\eta_k+1-j)} \frac{(2\nu_k t)^l}{l!}  \right]
\end{equation*}
proves the statement.

\subsection{Proof of Lemma \ref{lemma:irreducibility}}
\label{app:control}


In order to rely on a linear control problem, we decouple the two populations and treat the non-linear interactions in a second step as it is done in \cite{Locherbach2019} for the continuous-time framework. Let us rewrite the numerical scheme \eqref{eq:split_scheme_LT2} as given by the one-step mapping $\psi_\Delta$ defined by
\begin{equation}
    \psi_\Delta[\xi] (x) = e^{A\Delta}\left( x + \Delta  B(x) + \frac{\sqrt{\Delta}}{\sqrt N} \sigma(x)\xi \right) = 
    \begin{pmatrix}
    \psi_\Delta[\xi] (x)_1\\
    \psi_\Delta[\xi] (x)_2
    \end{pmatrix},
\end{equation}
where 
\begin{equation}
    \psi_\Delta[\xi] (x)_k = e^{A_{\nu_k}\Delta} x^k + \left(\Delta c_k f_{k+1}(x^{k+1,1}) + \frac{\sqrt{\Delta}}{\sqrt N} \frac{c_k}{\sqrt{p_{k+1}}} \sqrt{f_{k+1}(x^{k+1,1})} \xi^{k+1}  \right) b_k,
\end{equation}
with $b_k = e^{A_{\nu_k}\Delta}(0,\dots,0,1)^T = \left( \frac{\Delta^{\eta_k}}{\eta_k!}, \frac{\Delta^{\eta_k-1}}{(\eta_k-1)!}, \dots, 1 \right)^T \in \mathbb{R}^{\eta_k+1}$. Now let us study the following discrete dynamical systems: $x^k(0) = x^k$ and for all $t\in \mathbb{N}$,
\begin{equation}\label{eq:discrete:control:system}
    x^k(t+1) = e^{A_{\nu_k}\Delta} x^k(t) + b_k u^k(t+1),
\end{equation}
where $(u^k(t))_{t\in \mathbb{N}^*}$ is a sequence of real numbers that will be specified below. This system is controllable as soon as $b_k, e^{A_{\nu_k}\Delta}b_k,\dots, e^{\eta_k A_{\nu_k}\Delta}b_k$ are linearly independent (see Theorem 6.D1 in \cite{chen1998linear}).
For all $j=0,\dots, \eta_k$, we have
\begin{equation}
    e^{jA_{\nu_k}\Delta}b_k = \left( \frac{((j+1)\Delta)^{\eta_k}}{\eta_k!}, \frac{((j+1)\Delta)^{\eta_k-1}}{(\eta_k-1)!}, \dots, 1 \right)^T.
\end{equation}
Yet, $\{1, X, \dots, X^{\eta_k}/\eta_k!\}$ is a basis of the vector space of polynomials with degree at most $\eta_k$ which ensures linear independence. The controllability of the system means that for all $x^k, y^k\in \mathbb{R}^{\eta_k+1}$, there exists some sequence of real numbers $(u^k(t))_{t=1,\dots, {\eta^*}+1}$ such that $x^k({\eta^*}+1)=y^k$ where $x^k$ is inductively defined by \eqref{eq:discrete:control:system}. In the following, we use the notation $x(t) = (x^1(t),x^2(t))^T$.
 
Now, let $x$ and $y$ be as in the statement of Lemma \ref{lemma:irreducibility} and denote $x=(x^1,x^2)^T$ and $y=(y^1,y^2)^T$. According to the first step of the proof, let $(u^k(t))_{t=1,\dots, {\eta^*}+1}$ be such that $x^k({\eta^*}+1)=y^k$ and define, for all $t=1,\dots, {\eta^*}+1$, 
\begin{equation}
    \xi^k(t) = \frac{u^k(t) - \Delta c_{k+1}f_k(x^{k,1}(t))}
    {\frac{\sqrt{\Delta}}{\sqrt N} \frac{c_k}{\sqrt{p_{k+1}}} \sqrt{f_k(x^{k,1}(t))}},
\end{equation}
in such a way that
\begin{equation}
    u^k(t) = \Delta c_{k+1} f_k(x^{k,1}(t)) + {\frac{\sqrt{\Delta}}{\sqrt N} \frac{c_k}{\sqrt{p_{k+1}}} \sqrt{f_k(x^{k,1}(t))}}\xi^k(t).
\end{equation}
Substituting $u^k(t+1)$ in \eqref{eq:control:modulus:1} and denoting $\xi_t=\xi(t)$, yields $x^k(t+1) = \psi_\Delta[\xi_{t+1}](x(t))_k$ and thus
\begin{equation}
    y = x(\bar{\eta}+1) = \underbrace{\left(\psi_\Delta[\xi_{\bar{\eta}+1}] \circ \dots \circ \psi_\Delta[\xi_1] \right)}_{\eta^*+1} (x),
\end{equation}
which proves the result.
\end{document}